\documentclass[a4paper, 11pt, twoside, leqno]{amsart}

\usepackage[T1]{fontenc}
\usepackage[utf8x]{inputenc}
\usepackage[english]{babel}
\usepackage{amsmath, amsthm, amssymb, stmaryrd}
\usepackage{mathrsfs, esint, bbm}   
\usepackage{fancyhdr}               
\usepackage{enumitem}               
\usepackage[all,cmtip]{xy}          
\usepackage{graphicx}
\usepackage{hyperref}

\newtheorem{theorem}{Theorem}            

\newtheorem{lemma}[theorem]{Lemma}
\newtheorem{prop}[theorem]{Proposition}

\newtheorem{mainthm}{Theorem}          

\theoremstyle{definition}              

\theoremstyle{remark}                  
\newtheorem{step}{Step}

\newtheorem{remark}{Remark}

\newcommand{\N}{\mathbb{N}}                 
\newcommand{\Z}{\mathbb{Z}}                 
\newcommand{\R}{\mathbb{R}}                 
\renewcommand{\H}{\mathscr{H}}              

\newcommand{\T}{\mathrm{T}}
\newcommand{\gf}{\mathfrak{g}}
\renewcommand{\P}{\mathbb{P}}

\renewcommand{\a}{\mathbf{a}}
\renewcommand{\b}{\mathbf{b}}
\newcommand{\be}{\mathbf{e}}
\newcommand{\db}{\mathbf{d}}
\renewcommand{\d}{\mathrm{d}}               
\newcommand{\D}{\mathrm{D}}

\newcommand{\GG}{\mathscr{G}}
\newcommand{\LL}{\mathcal{L}}

\newcommand{\Sh}{\mathscr{S}}
\newcommand{\UU}{\mathscr{U}}

\newcommand{\epsi}{\varepsilon}
\newcommand{\eps}{\epsi}

\newcommand{\Harm}{\mathrm{Harm}}

\DeclareMathOperator{\dist}{dist}                                   
\DeclareMathOperator{\sign}{sign}                                   
\DeclareMathOperator{\ind}{ind}                                     
\DeclareMathOperator{\tang}{tan}
\DeclareMathOperator{\curl}{\curl}                                  
\renewcommand{\div}{\mathrm{div}\,}                                 
\newcommand{\Vg}{\mathrm{vol}_g}
\renewcommand{\O}{\mathrm{O}}
\newcommand{\abs}[1]{\left| #1 \right|}                             
\newcommand{\norm}[1]{\left\| #1 \right\|}                          

\newcommand{\intr}{\mathrm{intr}}
\newcommand{\extr}{\mathrm{extr}}

\usepackage{color}
\usepackage[normalem]{ulem} 

\begin{document}

\title[Vortices for the extrinsic Ginzburg-Landau flow on surfaces]{Motion of vortices for the extrinsic Ginzburg-Landau flow for vector fields on surfaces}
\author{Giacomo Canevari}
\address{Dipartimento di Informatica, 
  Universit\`a di Verona,
  Strada le Grazie 15, \mbox{37134} Verona, Italy.}
\email[G. Canevari]{giacomo.canevari@univr.it}
\author{Antonio Segatti}
\address{Dipartimento di Matematica ``F. Casorati'',
  Universit\`a di Pavia,
  Via Ferrata 1, \mbox{27100} Pavia, Italy.}
\email[A. Segatti]{antonio.segatti@unipv.it}

\date{\today}

\dedicatory{\large{Dedicated to Maurizio Grasselli on the occasion of his 60th anniversary, \\ with friendship and admiration}}

\begin{abstract}
We consider the gradient flow of a Ginzburg-Landau functional of the type
\[
F_\eps^{\extr}(u):=\frac{1}{2}\int_M \abs{D u}_g^2 + \abs{\Sh u}^2_g
+\frac{1}{2\eps^2}\left(\abs{u}^2_g-1\right)^2\Vg
\]
which is defined for tangent vector fields
(here $D$ stands for the covariant derivative) on a closed surface~$M\subseteq\R^3$
and includes extrinsic effects via the shape operator $\Sh$
induced by the Euclidean embedding of~$M$. The functional depends on the small parameter
$\eps>0$. When $\eps$ is small it is clear from the structure of the Ginzburg-Landau functional that $\abs{u}_g$  ``prefers'' to be close to $1$.  However, due to the incompatibility for vector fields on $M$ between the Sobolev regularity and the unit norm constraint, when $\eps$ is close to $0$, it is expected that a finite number of singular points (called vortices) having non-zero index emerges (when the Euler characteristic is non-zero). This intuitive picture has been made precise in the recent work by R. Ignat \& R. Jerrard \cite{JerrardIgnat_full}. 
In this paper we are interested the dynamics of vortices generated by $F_\eps^{\extr}$. 
To this end we study the behavior when $\eps\to 0$ of the solutions of the (properly rescaled) gradient flow of $F_\eps^{\extr}$. 
In the limit $\eps\to 0$ we obtain the effective dynamics of the vortices. The dynamics, as expected, is influenced by both the intrinsic and extrinsic properties of the surface~$M\subseteq\R^3$. 


 
\end{abstract}

\maketitle
\tableofcontents


\section{Introduction and main results}
\label{sec:intro}

In this paper we continue the analysis of the dynamics of Ginzburg-Landau 
vortices on a two-dimensional Riemannian manifold started in~\cite{Ag-Reno}
by considering the effects of the extrinsic geometry of the manifold. 
More precisely, let~$M$ be a closed, oriented two-dimensional Riemannian manifold isometrically embedded in~$\mathbb{R}^3$, with genus $\gf = \frac{2-\chi(M)}{2}\geq 0$, $\chi(M)$ being the Euler Characteristic of~$M$. We denote with $g$ the metric induced by the embedding and with $D$ the Levi Civita connection of the metric~$g$ on~$M$. 
Let us denote with~$N$ the unit normal to~$M$ chosen accordingly with the orientation. The role of the extrinsic geometry, which is related to the fact that the manifold~$M$ is embedded in~$\R^3$, is played by the so-called shape operator of~$M$. 
The shape operator is the self-adjoint operator 
\begin{equation}
\label{eq:shape}
\Sh: TM\to TM, \qquad \Sh v: = -\nabla_{v}N,
\end{equation} 
where~$\nabla$ is the connection with respect to the standard metric in~$\R^3$.
Being self-adjoint, we can define the operator 
\[
\Sh^2 :TM \to TM, \qquad \left(\Sh^2 u, \, v\right)_g := (\Sh u, \, \Sh v)_g\quad \textrm{for any } u,v\in TM.
\]
For any $\eps>0$ we consider the following evolution for the vector field $u_\eps\colon M\times [0, \, +\infty)\to T M$ 
\begin{equation}
\label{eq:GL_intro}
\begin{cases}
\displaystyle\frac{1}{\vert \log\eps\vert}\partial_t u_\eps -\Delta_g u_\eps + \Sh^2 u_\eps+\frac{1}{\eps^2}\left(\vert u_\eps\vert^2-1\right)u_\eps = 0,\,\,\,\,\hbox{ a.e. in }M\times (0,T),\\[.4cm]
\displaystyle u_\eps(x,0) = u_0^{\eps}\,\,\,\,\hbox{ a.e. in }M.
\end{cases}
\end{equation}
The operator $\Delta_g:= -D^* D$ is the so-called rough Laplacian ($D^*$ is the adjoint of the covariant derivative, see \cite[Section 2]{Ag-Reno} for the definition).
The evolution~\eqref{eq:GL_intro} is the gradient flow of the following  functional (called the extrinsic Ginzburg-Landau functional, see \cite{ssvM3AS})
\begin{equation}
\label{eq:extrinsic_energy}
F_\eps^{\extr}(u):=\frac{1}{2}\int_M \abs{D u}_g^2 + \abs{\Sh u}^2_g+\frac{1}{2\eps^2}\left(\abs{u}^2_g-1\right)^2\Vg.
\end{equation}
Given a smooth vector field $u$, we smoothly extend it to a vector field on $\R^3$ and we call this extension $\tilde{u}$. For $p\in M$, we define the surface gradient of $u$ as 
\begin{equation}
\label{eq:surface}
\nabla_s u(p) := \nabla \tilde{u}(p)\,P(p),
\end{equation} 
where $P(p):=(\textrm{Id} - N(p)\otimes N(p))$ is the orthogonal projection on $T_p M$. This differential operator is well defined and does not depend on the chosen extension $\tilde{u}$. Moreover, in general, $\nabla_s u\neq D u =P \nabla \tilde{u} $. Finally, thanks to the Gauss Formula
\[
\nabla_ v u = D_v u +   (\Sh u, v)_g N \qquad 
\textrm{for any }v\in T_p M, \ p\in M,
\]
we readily get 
\[
\nabla_s u[v] =\nabla_v u = D_v u +  (\Sh u,  v)_gN
\qquad \textrm{for any } v\in T_p M, \ p\in M.
\]
Thus, since the above decomposition is orthogonal, we deduce
\begin{equation}
\label{eq:extrinsic_distorsion}
\abs{\nabla_s u}^2 = \abs{D u}^2 + \abs{\Sh u}^2
\end{equation}
at any point of~$M$.
Therefore, the extrinsic Ginzburg-Landau energy~\eqref{eq:extrinsic_energy} rewrites as 
\begin{equation}
\label{eq:extrinsic_energy_bis}
F_\eps^{\extr}(u):=\frac{1}{2}\int_M \abs{\nabla_s u}^2+\frac{1}{2\eps^2}\left(\abs{u}^2_g-1\right)^2\Vg.
\end{equation}
This type of energy enters in the modeling of the so-called Nematic Shells. A Nematic Shell is a thin film of nematic liquid crystal coating a rigid and curved substrate $M$, typically modeled as a two-dimensional closed and oriented surface (see, e.g., \cite{LubPro92}, \cite{Nelson02}, \cite{Straley71}). 
In this framework, the preferred direction of the molecular alignment is usually described by a unit-norm vector field tangent to~$M$ and the energy $F_\eps^{\extr}$ is a 
relaxation (relaxing in particular the unit-norm constraint) of the elastic energy (see \cite{NapVer12E}, \cite{NapVer12L}, \cite{ssvPRE}, \cite{ssvM3AS})
\begin{equation}
\label{eq:one_constant}
F_0^\extr(u) := 
\begin{cases}
 \dfrac{1}{2}\displaystyle\int_{M}\abs{\nabla_s u}^2 \Vg
  \qquad &\textrm{if } \abs{u}=1 \ \textrm{a.e. in } M\\
 +\infty\qquad &\textrm{otherwise}
\end{cases}
\end{equation}
A feature of the functional~$F_0^\extr$ is to incorporate,
via the surface gradient~$\nabla_s$,
the extrinsic geometry of~$M$ in the local distorsion element. 
From the analysis standpoint,
due to the incompatibility between the unit-norm constraint and the Sobolev regularity, this energy is well-defined and not identically equal to~$+\infty$
only if~$\chi(M)=0$ (see \cite{ssvM3AS} and \cite{AGM-VMO}).
When $\chi(M)=0$ one can prove (see e.g.~\cite{gamma-discreto}) that $F_\eps^{\extr}\xrightarrow{\eps\to 0} F_0^\extr$
in the sense of $\Gamma$-convergence. 
When $\chi(M)\neq 0$ the analysis is more involved as one has to deal with the emergence of vortices. In this case, Ignat \& Jerrard obtained (see \cite{JerrardIgnat_full}) a $\Gamma$-convergence result for $F_\eps^{\extr}$. After a proper scaling, what emerges in the limit $\eps\to 0$ is a renormalized energy that takes into account the interactions between the vortices and between the vortices and the geometry of the surface, both intrinsic and extrinsic. 
This energy is given by the sum of the renormalized energy obtained in the intrinsic case and of a new, purely extrinsic, term. 

As in the analysis of the intrinsic case, a central role 
is played by the so-called vorticity. 
We recall that that the vorticity~$\omega(u)$ of a smooth vector field~$u$
is defined as the $2$-form
\begin{equation}
\label{eq:vorticity_intro}
 \omega(u):= \d j(u) + \kappa \, \Vg,
\end{equation}
where~$\kappa$ is the Gauss curvature of~$M$ and~$j(u)$ is the $1$-form 
\begin{equation}
\label{eq:prejac_intro}
 j(u):= (D u, iu)_g.
\end{equation}
The symbol~$i$ denotes an almost complex structure on~$M$, that is,
an operator~$TM\to TM$ which restricts to a linear map
on each tangent plane~$T_p M$ and satisfies~$i^2+1 = 0$
(where~$1$ denotes the identity on~$TM$). 
In the Euclidean setting, when $u\colon\R^2\to\R^2$, 
the vorticity~$\omega(u)$ essentially reduces to the 
Jacobian determinant of~$u$, i.e.~$\omega(u) = 2(\det \nabla u) \, \d x$. 

Let~$W^{1,p}_{\tang}(M)$ be the set of maps~$u\in W^{1,p}(M, \, \R^3)$
such that~$u\cdot N = 0$ a.e. (see e.g.~\cite{ssvM3AS}, \cite{JerrardIgnat_full}
 and \cite{Ag-Reno} for the details). Since 
\[
 F_\eps^{\extr}(u)
  \ge F_\eps(v)
  := \frac{1}{2}\int_M \abs{D u}_g^2 +\frac{1}{2\eps^2}\left(\abs{u}^2_g-1\right)^2\Vg \qquad \textrm{for any } u\in W^{1,2}_{\tang}(M),
\]
sequences $u_\eps$ satisfying a energy bound such as $F_\eps^{\extr}(u_\eps) \leq  n\pi\abs{\log\eps} + C$
(for some constants~$n$ and $C$ that do not depend on~$\eps$) 
have compact vorticities. More precisely (see \cite{JerrardIgnat_full}),
there exist distinct points $a_1$, \ldots, $a_n$ in~$M$, 
integers~$d_1$, \ldots, $d_n$,
a form~$\xi\in\Harm^1(M)$ and a (non-relabelled) subsequence
such that, denoting with $\mathbb{P}_H$ the $L^2$-projection 
onto the space~$\Harm^1(M)$ of harmonic $1$-forms,
\begin{align*}
&\omega(u_\eps)\xrightarrow{\eps\to 0} 2\pi \sum_{j=1}^n d_j\delta_{a_j}\qquad \hbox{ in } W^{-1,p}(M) \quad
 \textrm{for any } p\in (1, \, 2) \\
& \mathbb{P}_H j(u_\eps) \xrightarrow{\eps\to 0}\xi
\end{align*}
The integers~$d_j$ represent the charges of
topological singularities which arise at the points~$a_j$.
They satisfy
\[
 \sum_{j=1}^n d_j = \chi(M),
\]
in accordance with the Poincar\'e-Hopf theorem.
The harmonic $1$-form $\xi\in \Harm^1(M)$, 
which emerges when the genus $\gf\neq 0$, depends non-trivially
on~$\a:=\left(a_1, \, \ldots, \, a_n\right)$
and~$\db:=(d_1, \, \ldots, \, d_n)$, as it must satisfy
the constraint $\xi\in\mathcal{L}(\a, \, \db)$
where~$\mathcal{L}(\a, \, \db)$ is a suitable subset of~$\Harm^1(M)$
(see~\cite[Section~2.1]{JerrardIgnat_full} and \cite[Section~3]{Ag-Reno}
for the definition of~$\mathcal{L}(\a, \, \db)$.)

In case the~$u_\eps$'s are minimizers of~$F_\eps^{\extr}$,
the results of~\cite{JerrardIgnat_full} provide a more detailed description
of the asymptotic behaviour as~$\eps\to 0$. First, the number of 
the singular points~$a_1$, \ldots, $a_n$ that arise in the limit is
exactly equal to~$n = \abs{\chi(M)}$, and each singularity 
has degree~$d_j = \sign(\chi(M))$. Second, not only the 
vorticities~$\omega(u_\eps)$ but also the minimisers~$u_\eps$
themselves converge, strongly in~$W^{1,p}(M)$ for any~$p < 2$,
up to extraction of a subsequence. The limit field~$u_0$ 
belongs to~$W^{1,p}(M)$ for any~$p < 2$, 
is smooth in~$M\setminus \{a_1, \, \ldots, \, a_n\}$, 
and satisfies~$\abs{u_0} = 1$ a.e. Moreover, $u_0$ can be written
in the form
\begin{equation} \label{u0}
 u_0 = e^{i\theta} u^* := (\cos\theta) u^* + (\sin\theta) i u^*
\end{equation}
The vector field~$u^*$ in the right-hand side of Equation~\eqref{u0}
is the so-called \emph{canonical harmonic field}
associated with~$(\a, \, \db, \, \xi)$.
That means, $u^*$ is a $W^{1,1}$-vector field on~$M$
that satisfies~$\abs{u^*}= 1$ a.e.~and~$j(u^*) = \d^*\Psi + \xi$,
where the $2$-form~$\Psi = \Psi[\a, \, \db]$ is uniquely determined 
(up to an additive constant) as a solution of 
\[
 -\Delta \Psi = 2\pi \sum_{j=1}^n d_j \delta_{a_j} -\kappa \, \Vg.
\]
Such a vector field~$u^*$ exists if and only if~$\xi$
satisfies the constraint~$\xi\in\mathcal{L}(\a, \, \db)$;
if it exists, $u^*$ is essentially unique
(i.e., it is unique up to a constant rotation; 
see~\cite[Theorem~2.1]{JerrardIgnat_full}).
The limit field~$u_0$ differs from the canonical harmonic
vector field~$u^*$ by a rotation~$e^{i\theta}$.
The rotation angle~$\theta\in H^1(M)$ in~\eqref{u0}
must be chosen so as to minimize the functional
\begin{equation}
\label{eq:extr_en_intro}
 \GG(\a, \, \db, \, \xi, \, \theta) 
 := \frac{1}{2}\int_{M}\abs{\d \theta}^2_{g} +\abs{\Sh(e^{i\theta}u^*)}^2_{g}\Vg
\end{equation}
and~$u^* = u^*[\a, \, \db, \, \xi]$ is as above.
The functional~$\GG$ accounts for the ``extrinsic''
contributions to the energy, i.e., the effects of the shape operator.
Finally, the results of~\cite{JerrardIgnat_full} provide an energy expansion
(by $\Gamma$-convergence); in particular,
minimizers~$u_\eps$ of~$F_\eps^\extr$ satisfy
\begin{equation} \label{eq:F-Gamma}
 F_\eps^{\extr}(u_\eps) = \pi n\abs{\log\eps}
  + \widetilde{W}(\a, \, \db, \, \xi) + n\gamma + \mathrm{o}(1)
  \qquad \textrm{as }\eps\to 0.
\end{equation}
The constant $\gamma>0$, analogously to the intrinsic case,
is the so-called core energy and accounts for the energy
in a suitably small neighbourhood of each singular point $a_j$. 
The energy $\widetilde{W}$ is the so-called \emph{renormalized energy},
that accounts for the interactions among the singularities.
The renormalised energy accounts for both intrinsic and 
extrinsic contributions. It is given by
\begin{equation}
\label{eq:renoJI}
\widetilde{W}(\a, \, \db, \, \xi) := W^{\intr}(\a, \, \db, \, \xi)
 + \inf_{\theta\in H^1(M)}\GG(\a, \, \db, \, \xi, \, \theta)
\end{equation}
The intrinsic component~$W^\intr$ is defined as
\begin{equation}
 \label{eq:renoJIintr}
  W^{\intr}(\a, \, \db, \, \xi) := \lim_{\rho\to 0}
  \left(\frac{1}{2} \int_{M\setminus\cup_{j=1}^n B_\rho(a_j)} 
  \abs{ Du^*}^2 \Vg - \pi \sum_{j=1}^n d_j^2 \abs{\log\rho} \right)
\end{equation}
where~$u^* = u^*[\a, \, \db, \, \xi]$ is the (essentially unique)
canonical harmonic field for~$(\a, \, \db, \, \xi)$. 
The extrinsic contributions are accounted for by the functional~$\GG$,
as defined in~\eqref{eq:extr_en_intro}. While it is known that~$W^\intr$
is well-defined and smooth (see e.g.~\cite[Proposition~2.4]{JerrardIgnat_full}),
it is unclear whether~$\widetilde{W}$ is smooth, 
for the functional~$\GG(\a, \, \db, \, \xi, \, \cdot)$ is not 
convex and it might have several minimizers.

When dealing with the analysis of the solutions of \eqref{eq:GL_intro} as $\eps\to 0$, it is key (see \cite{SS-GF}, \cite{Ag-Reno}) to obtain an estimate of the type
\begin{equation}
\label{eq:lb_grad_desiderio}
\begin{split}
   \liminf_{\eps\to 0} \frac{\abs{\log\eps}}{2}
    \int_{M}\abs{ -\Delta_g u_\eps(t) 
     + \frac{1}{\eps^2}(\vert u_\eps(t)\vert^2-1)u_\eps(t)}^2_{g}\,\Vg
     \geq  \frac{1}{2\pi} \sum_{j=1}^n
     \abs{\nabla_{a_j} \widetilde{W}(\a(t),\mathbf{d},\xi(t))}_g^2
  \end{split}
\end{equation}
Unfortunately, for the moment, such an estimate is \emph{not} available. 
The reason is twofold. The first reason, as we have seen,
is that the minimization of~$\GG$ introduces a source of 
non-smoothness in the energy~$\widetilde{W}$ and makes 
it difficult to characterize the gradient of~$\widetilde{W}$.
The second reason is the following. 
Given a sequence of solutions~$u_\eps$ of~\eqref{eq:GL_intro}
that satisfy suitable energy estimates, 
we can prove, as in the intrinsic case (see \cite{Ag-Reno}), 
that the sequence~$(u_\eps(t))_{\eps> 0}$ is compact
for a.e.~$t$ in a suitable subinterval~$[0, \, T^*)$
More precisely, for a.e.~$t\in [0, \, T^*)$ there exists a subsequence
$\eps_h(t)\to 0$, possibily depending on~$t$, such that
$u_{\eps_h(t)}$ converges (strongly in~$W^{1,p}_{\tan}(M)$ for~$p<2$)
to a vector field of the form
\[
 u_0(t) = e^{i\theta(t)} u^*(t)
\]
As in~\eqref{u0}, $u^*(t)$ is a canonical harmonic field
with singularities at~$\a(t) = (a_1(t), \, \ldots, \, a_n(t))$
of degrees~$\db = (d_1, \, \ldots, \, d_n)$ 
and harmonic component~$\xi(t)\in \mathcal{L}(\a(t), \, \db)$,
while~$\theta(t)\in C^1(M)$ is a critical point of the
functional~$\GG(\a(t), \, \db, \, \xi(t), \, \cdot)$.
Unfortunately, it is not clear if~$\theta(t)$ is indeed 
a minimizer of~$\GG$ and not a critical point only.
For these reasons, we work with the renormalized energy 
\begin{equation}
\label{eq:reno_intro}
  W(\a, \, \db, \, \xi, \, \theta) := W^{\intr}(\a, \, \db, \, \xi) 
  + \GG(\a, \, \db, \, \xi, \, \theta),
\end{equation}
where  $(\a,\db)\in \mathscr{A}^n,\, \xi\in \Harm^1(M),\, \theta\in H^1(M).$
In Section \ref{sec:gradient} we prove that such an energy
is smooth in~$\a$ for any $\theta\in H^1(M)$ and we characterize 
its gradient with respect to~$\a$ when $\theta$ is a critical point of~$\GG$. 
Defining the gradient of~$W$ with respect to the variable~$\a$ 
requires some care, for~$\a$, $\db$ and~$\xi$ are not independent variables ---
they are related to each other by the constraint~$\xi\in\mathcal(\a, \, \db)$;
see~\eqref{gradW} for more details.
As it turns out, once we prescribe an initial condition~$\xi(0) = \xi^0$,
then the evolution of the component~$\xi(t)$ 
is completely determined by the equation for~$\a(t)$
and the constraint~$\xi(t)\in\mathcal{L}(\a(t), \, \db)$.

As in \cite{Ag-Reno}, we consider solutions of~\eqref{eq:GL_intro} 
with well-prepared initial conditions. 
We say that the initial condition~$u^0_\eps$ is well-prepared 
if satisfies the following set of hypothesis.
Given~$n\in\Z$, $n\geq 1$,
we consider a~$n$-uple of \emph{distinct}
points~$\a^0 = (a^0_1, \, \ldots, \, a^0_n)$ in~$M$
and a~$n$-uple of integers~$\db = (d_1, \, \ldots, \, d_n)$
such that
\[
 \sum_{j=1}^n d_j = \chi(M), \qquad \abs{d_j} = 1 \quad \textrm{for any } j.
\]
Moreover, we fix a harmonic $1$-form $\xi^0$ such that
$\xi^0\in \mathcal{L}(\a^0, \, \db)$ and $\theta^0\in H^1(M)$.
We consider vector fields~$u^0_{\eps}\in H^1_{\tang}(M)$ 
such that 
\begin{align}
 & \omega (u^{0}_\eps)\xrightarrow{\eps \to 0} 2\pi \sum_{j=1}^n d_j \delta_{a_j^{0}}\, \,\,\,\text{ in } W^{-1,p}(M)\label{eq:initial_vorticity_intro}\\
 & F_\eps^{\extr}(u_\eps^{0}) = \pi n \abs{\log\eps}
 + W(\a^0, \, \db, \, \xi^0, \, \theta^0) + n\gamma + \mathrm{o}_{\eps\to 0}(1),\label{eq:well_prepared_intro} \\
 & \norm{u^0_\eps}_{L^\infty(M)} \leq 1. \label{eq:initial_Linfty_intro}
\end{align} 
%
%
%
%
The main result of the paper is the following
\begin{mainthm}
\label{th:main1}
Let $(M,g)$ be a closed oriented two-dimensional Riemannian manifold isometrically embedded in $\mathbb{R}^3$. 
Let $u_\eps$ be a sequence of solutions of~\eqref{eq:GL_intro}
with $u_\eps^{0}\in H^1_{\tang}(M)$ a sequence of well-prepared initial conditions. 
Then there exists a time $T^*\in (0,T)$ and a non increasing function 
$\varphi:[0,T^*)\to \R$ such that for any $t\in (0,T^*)$
\begin{align*}
& \omega(u_\eps(t))\xrightarrow{\eps\to 0}  2\pi\sum_{j=1}^n d_j\delta_{a_j(t)}\qquad \hbox{ in } W^{-1,p}(M) \quad
 \textrm{for any } p\in (1, \, 2) \\
& \mathbb{P}_H\left(j(u_\eps(t))\right) \xrightarrow{\eps\to 0}\xi(t)\\
& F_\eps^{\extr}(u_\eps(t))-n\pi\abs{\log\eps}-n\gamma \xrightarrow{\eps\to 0}\varphi(t)
\end{align*}
where $\a(t):=(a_1(t), \, \dots, \, a_n(t))$ is 
an~$n$-uple of distinct points such that~$\a\in H^1(0, \, T^*; \, M^n)$, $\a(0) = \a^0$,
and $\xi\in H^{1}(0,T^*;\Harm^1(M))$ with 
\begin{equation} \label{eq:xi_mainth}
 \xi(t)\in \mathcal{L}\left(\a(t), \, \db\right) 
  \quad \textrm{for any } t\in (0, \, T^*), 
  \qquad \xi(0) = \xi^0. 
\end{equation}
Moreover, there exists a measurable function
$\theta\colon(0, \, T^*)\to C^1(M)$ such that
\[
\partial_\theta \GG(\a(t), \, \db, \, \xi(t), \, \theta(t))=0
\qquad \textrm{for a.e.~} t\in (0, \, T^*)
\]
Finally, for a.e.~$s$, $t$ with~$0\le s<t<T^*$, there holds
\begin{equation}
\label{eq:gradflowW_mainth}
\begin{split}
-\frac{\d}{\d t}\varphi(t) &\ge \frac{\pi}{2}\abs{\a'(t)}^2_{g} + \frac{1}{2\pi}\abs{\nabla_{\a}W(\a(t), \, \db, \, \xi(t),\,\theta(t))}_g^2,\\
 \varphi(t) &\ge W(\a(t),\,\db,\,\xi(t),\,\theta(t)), \\
  \varphi(0) &= W(\a^0,\,\db,\,\xi^0,\,\theta^0).
\end{split}
\end{equation}
\end{mainthm}
Note that~\eqref{eq:gradflowW_mainth} is indeed a weak formulation
of the gradient flow for~$W$, with respect to
the variable~$\a$. Indeed, assume~$\theta$ is absolutely continuous
with respect to time. As
\[
 \partial_\theta W(\a(t), \, \db, \, \xi(t), \, \theta(t)) = \partial_\theta \GG(\a(t), \, \db, \, \xi(t), \, \theta(t)) = 0
\]
for almost any time~$t$, we have
\begin{equation}
\label{eq:chain}
\begin{split}
-\frac{\d}{\d t}W(\a(t), \, \db, \, \xi(t), \, \theta(t)) 
&= \left(\nabla_\a W(\a(t), \, \db, \, \xi(t), \, \theta(t)), 
 \, -\frac{\d}{\d t}\a(t)\right)_g\\
& \le \frac{\pi}{2}\abs{\a'(t)}^2_{g} + \frac{1}{2\pi}\abs{\nabla_{\a}W(\a(t), \, \db, \, \xi(t),\,\theta(t))}_g^2,
\end{split}
\end{equation}
(The dependence of~$W$ on~$\xi$ is accounted for in~$\nabla_{\a} W$;
see~\eqref{gradW} for more details.)
Therefore, it turns out that 
\[
\frac{\d}{\d t}\varphi(t) \le \frac{\d}{\d t}W(\a(t), \, \db, \, \xi(t), \, \theta(t))
\]
for almost any $t\in (0, \, T^*)$, and thus
\[
 W(\a(t), \, \db, \, \xi(t), \, \theta(t)) = \varphi(t)
\]
for any $t\in [0,T^*)$. As a result we have that $\a$ is indeed a curve of maximal slope and in particular we obtain the evolution
\begin{equation} 
\label{eq:formal_limit}
\begin{cases}
 \displaystyle\frac{\d}{\d t} \a(t)
 = -\frac{1}{\pi}\nabla_{\a} W(\a(t), \, \db, \, \xi(t), \, \theta(t))
 \qquad &\textrm{for any } t\in (0,T^*),\\[.4cm]
 \xi(t)\in \mathcal{L}(\a(t), \, \db)
  \qquad &\textrm{for any } t\in (0,T^*)\\[.2cm]
 \partial_\theta \GG(\a(t), \, \db, \, \xi(t), \, \theta(t)) = 0
 \qquad &\textrm{for any } t\in (0,T^*)\\[.2cm]
 \a(0) = \a^0, \qquad \xi(0) = \xi^0.
\end{cases}
\end{equation}
The system above has some interesting features as it couples an ODE for the evolution of $\a$-i.e. the gradient flow of $W$-with the nonlinear elliptic equation $\partial_\theta \GG(\a,\db,\xi,\theta)=0$. 
In this sense \eqref{eq:formal_limit} is a {\itshape quasi stationary } system. 
In particular, it seems difficult and challenging to obtain some regularity with respect to time for $\theta$ as no control on $\partial_t \theta$ is available from the equations.

The paper is organized as follows. In the next section, we study in detail the differentiability properties of the 
renormalized energy~$W$ and characterize its gradient with respect to the position of vortices. 
Then, in Section~\ref{sec:dyna_vortex}, we study the asymptotic behavior when $\eps\to 0$ of the solutions of the  Ginzburg-Landau equation \eqref{eq:GL_intro} and prove Theorem \ref{th:main1}.

\section{The renormalized energy and its gradient}
\label{sec:gradient}

\subsection{The functional \texorpdfstring{$\GG$}{G} and its gradient}
\label{sect:G}

In this section, we consider the functional~$\GG 
= \GG(\a, \, \db, \, \xi, \, \theta)$ defined 
in~\eqref{eq:extr_en_intro}. Our goal is to compute
the gradient of~$\GG$ with respect to  
the variable~$\a$. From now on, we abuse of notation
and regard~$\GG$ as a functional of a vector field~$u$ and a function~$\theta$
--- that is, we define~$\GG\colon L^2_{\tan}(M)\times H^1(M)\to\R$ as
\begin{equation} \label{G}
 \GG(u, \, \theta) := \frac{1}{2}\int_M\left(
 \abs{\d\theta}^2_g + \abs{\Sh(e^{i\theta}u)}^2_g\right)\Vg,
\end{equation}
Here, $\Sh\colon TM\to TM$ is the shape operator of~$M$, 
as given by~\eqref{eq:shape}. With the notation used in Section~\ref{sec:intro},
we have~$\GG(\a, \, \db, \, \xi, \, \theta)
= \GG(u^*[\a, \, \db, \, \xi], \, \theta)$
where~$u^*[\a, \, \db, \, \xi]$ denotes a suitably chosen canonical harmonic field
for~$(\a, \, \db, \, \xi)$ --- see~\eqref{pointvalue} below.
For any~$u\in L^2_{\tan}(M)$, the functional~$\GG(u, \, \cdot)$
is Fr\'echet-differentiable. Moreover, the definition of~$\GG$ 
immediately implies that
\begin{equation} \label{G-symmetry}
 \GG(e^{i\kappa}u, \, \theta - \kappa) = \GG(u, \, \theta)
\end{equation}
for any $(u, \, \theta)\in L^2_{\tang}(M)\times H^1(M)$
and any constant~$\kappa\in\R$.

We recall some notation from~\cite{JerrardIgnat_full, Ag-Reno}.
We recall the definition of the admissible class
\begin{equation} \label{admissible}
 \mathscr{A}^n := \left\{(\a, \, \db)\in M^n\times\Z^n\colon
 a_j \neq a_k \ \textrm{ for any } j\neq k \ \textrm{ and } \
 \sum_{j=1}^n d_j = \chi(M)\right\}.
\end{equation}
For any~$(\a, \, \db)\in\mathscr{A}^n$, let~$\Psi = \Psi[\a, \, \db]$ 
be the unique $2$-form that satisfies
\begin{equation} \label{Phiad}
 \begin{cases}
  -\Delta \Psi = 2\pi \displaystyle\sum_{k=1}^n d_k\,\delta_{\a_k}
   - \kappa\,\Vg \\
  \displaystyle\int_M \Psi = 0
 \end{cases}
\end{equation}
where~$\kappa$ is the Gauss curvature of~$M$. 
Given~$(\a, \, \db)\in\mathscr{A}^n$ and~$\xi\in\Harm^1(M)$,
a \emph{canonical harmonic field} for~$(\a, \, \db, \, \xi)$
is a vector field~$u^*\in W^{1,1}_{\tan}(M)$ such that
$\abs{u^*(x)} = 1$ for a.e.~$x\in M$ and
\begin{equation} \label{canonicalvf}
 j(u^*) = \d^*\Psi[\a, \, \db] + \xi
\end{equation}
where~$\Psi[\a, \, \db]$ is given by~\eqref{Phiad}.
Such a vector field~$u^*$ may or may not exist;
we call~$\mathcal{L}(\a, \, \db)$ the set of harmonic
$1$-forms~$\xi$ such that a canonical harmonic 
field for~$(\a, \, \db, \, \xi)$ exists.
(The set~$\mathcal{L}(\a, \, \db)$ is non-empty; a characterisation
is given in~\cite[Theorem~2.1]{JerrardIgnat_full}.)
However, if a canonical harmonic field for~$(\a, \, \db, \, \xi)$
exists, then it is unique up to a global rotation:
if~$u^*$, $w^*$ are canonical harmonic vector fields 
for the same set of parameters~$(\a, \, \db, \, \xi)$,
then there exists a constant~$\kappa\in\R$ such that
$w^* = e^{i\kappa} u^*$ 
\cite[Theorem~2.1]{JerrardIgnat_full}.

In order to restore uniqueness, we select a reference
canonical harmonic field by working locally in~$\a$
and imposing a renormalization condition.
Let~$\a^0 = (a_1^0, \, \ldots, \, a_n^0)\in M^n$
and~$\db = (d_1, \, \ldots, \, d_n)\in\Z^n$ be fixed,
such that~$(\a^0, \, \db)\in\mathscr{A}^n$.
Let~$\UU\subseteq M^n$ be an open neighbourhood
of~$\a^0$. By assumption, all the points~$a_j^0$ are distinct.
Therefore, by taking~$\mathscr{U}$ small enough 
we may also assume without loss of generality 
that $a_k\neq a_j$ for any $\a = (a_1, \, \ldots, \, a_n)\in\UU$ 
and any indices~$k$, $j$ with $k\neq j$.
Let~$b^0\in M$ be a point such that $b^0\neq a_k$ for any 
$\a = (a_1, \, \ldots, \, a_n)\in\UU$
and any index~$k$. Let~$w^0\in T_{b^0} M$. For any~$\a\in\UU$
and any~$\xi\in\mathcal{L}(\a, \, \db)$,
we define~$u^*[\a, \, \db, \, \xi]$ as the unique canonical harmonic vector
field for~$(\a, \, \db, \, \xi)$ that satisfies
\begin{equation} \label{pointvalue}
 u^*[\a, \, \db, \, \xi] (b^0) = w^0
\end{equation}
The condition~\eqref{pointvalue} makes sense, because
$u^*[\a, \, \db, \, \xi]$ is smooth in~$M\setminus\{a_1, \, \ldots, \, a_n\}$
(see~\cite[Theorem~2.1]{JerrardIgnat_full}).

The constraint~$\xi\in\mathcal{L}(\a, \, \db)$ allows us 
to write~$\xi$ as a function of~$\a$, in the following way.
Let us fix~$\xi^0\in\mathcal{L}(\a^0, \, \db)$.
If we take~$\UU$ small enough, then there exists a 
unique smooth map~$\Xi\colon\UU\to\Harm^1(M)$ such that
\begin{equation} \label{Xi}
 \Xi(\a)\in\LL(\a, \, \db) \quad
 \textrm{for any } \a\in\UU, \qquad 
 \Xi(\a^0) = \xi^0
\end{equation}
(see \cite[Lemma~3.3]{Ag-Reno}).
For ease of notation, we write~$u^*[\a] := u^*[\a, \, \db, \, \Xi(\a)]$.

\begin{prop} \label{prop:Gdiff}
 For any~$\theta\in H^1(M)$, the function
 \[
  \UU\to\R, \qquad \a\mapsto\GG(u^*[\a], \, \theta)
 \]
 is differentiable. Moreover, if~$\a\in\UU$
 and~$\theta\in H^1(M)$ is a critical point of~$\GG(u^*[\a], \, \cdot)$,
 then
 \[
  \nabla_{a_k} \GG(u^*[\a], \, \cdot) = 2\pi d_k \, i \,\nabla\theta(a_k)
 \]
 for any~$k\in\{1, \, \ldots, \, n\}$.
\end{prop}

The aim of this section is to prove Proposition~\ref{prop:Gdiff}.

\begin{lemma} \label{lemma:j}
 The map~$\a\mapsto j(u^*[\a])$ is continuous as 
 a map~$\UU\to L^p(M, \, T^*M)$ and differentiable
 as a map~$\UU\to W^{-1,p}(M, \, T^*M)$, for any~$p\in [1, \, 2)$.
\end{lemma}
\begin{proof}
 The map 
 \[
  F\colon\a\in\UU\mapsto 2\pi \sum_{k=1}^n d_k \, \delta_{a_k}
 \]
 is continuous as a map~$\UU\to (C^\alpha(M))^\prime$
 and differentiable as a map~$\UU\to (C^{1,\alpha}(M))^\prime$,
 for any~$\alpha\in (0, \, 1)$. By Sobolev embedding, it follows 
 that~$F$ is continuous as a map~$\UU \to W^{-1,p}(M)$
 and differentiable as a map~$\UU\to W^{-2,p}(M)$.
 Elliptic regularity theory (see, e.g.,\cite{gilb_trudi}) implies that
 the map~$\a\mapsto\Psi[\a, \, \db]$, defined as in~\eqref{Phiad},
 is continuous as a map~$\UU\to W^{1,p}(M)$ 
 and differentiable as a map~$\UU\to L^p(M)$. As~$\Xi$ is a smooth map
 (with values in the finite-dimensional space~$\Harm^1(M)$),
 the lemma follows.
\end{proof}

We can express the differential of~$j(u^*[\cdot])$ is a convenient way.
We denote as~$\D_{a_k} j(u^*[\a])$ the differential of the 
map~$j(u^*[\cdot])$ with respect to the variable~$a_k$ only, 
evaluated at the point~$\a$. 
By definition, $\D_{a_k} j(u^*[\a])$ is a linear operator
$T_{a_k} M\to L^p(M, \, T^*M)$.
We denote as~$\nu$ the unit vector field that is orthogonal
to all geodesic circles centered at~$a_k$ and points outward.
$\nu$ is well-defined and smooth in a neighbourhood of~$a_k$,
except at the point~$a_k$ itself.

\begin{lemma} \label{lemma:f}
 For any~$k\in\{1, \, \ldots, \, n\}$, $\a\in\UU$
 and~$v\in T_{a_k} M$, there exists a unique smooth function
 $f(\cdot, \, a_k, \, v)\colon M\setminus\{a_k\}\to\R$
 that belongs to~$L^p(M)$ for any~$p\in [1, \, 2)$ and satisfies
 \begin{align}
  \D_{a_k} j(u^*[\a]) [v] &= 
   2\pi d_k \, \d f(\cdot, \, a_k, \, v)  \label{nablaj} \\
  f(b^0, \, a_k, \, v) &= 0 \label{fzero}
 \end{align}
 (the point~$b^0\in M$ is the same as in~\eqref{pointvalue}).
 The function~$f(\cdot, \, a_k, \, v)$ is harmonic in~$M\setminus\{a_k\}$.
 Moreover, if~$w$ is a smooth vector field, defined in a neighbourhood
 of~$a_k$, such that~$w(a_k) = v$, then
 \begin{align} 
  f(\cdot, \, a_k, \, v) &= -\frac{1}{2\pi\dist_g(\cdot, \, a_k)} 
   (i\nu, \, w)_g + \O(\abs{\log\dist_g(\cdot, \, a_k)}) \label{localf} \\
  \d f(\cdot, \, a_k, \, v) &= \frac{1}{2\pi\dist_g(\cdot, \, a_k)}
   \left((i\nu, \, w)\,\nu^\flat
   + (\nu, \, w) \star\nu^\flat \right) + \O(\dist_g(\cdot, \, a_k)^{-1}) 
   \label{localdf}
 \end{align}
 as~$x\to a_k$.
\end{lemma}

Before we give the proof of Lemma~\ref{lemma:f}, we
introduce some notation. Let 
$G\colon \{(x, \, y)\in M\times M\colon x\neq y\}\to\R$
be the Green function for the Laplace-Beltrami operator~$-\Delta$ on~$M$, 
defined as
\begin{equation} \label{Green}
 \begin{cases} 
  -\Delta G(\cdot, \, y) = \delta_y - \dfrac{1}{\mathrm{Vol}(M)} 
    & \textrm{in } \mathscr{D}^\prime(M) \\[7pt]
  \displaystyle\int_M G(x, \, y) \Vg(x) = 0
 \end{cases}
\end{equation}
for any~$y\in M$. (Here~$\mathrm{Vol}(M)$ denotes the
area of~$M$, that is~$\mathrm{Vol}(M) := \int_M\Vg$.)
For any~$a\in M$ and any~$v\in T_a M$, we define the function
$\sigma(\cdot, \, a, \, v)\colon M\setminus\{a\}\to\R$ as
\begin{equation} \label{sigma_av}
  \sigma(x, \, a, \, v) := \left(\nabla_a G(x, \, a), \, v\right)
  \qquad \textrm{for } x\in M\setminus\{a\}
\end{equation}
Here~$\nabla_a G$ denotes the gradient of the Green function
with respect to its second argument. The function
$\sigma(\cdot, \, a, \, v)$ is harmonic, and hence smooth, 
in~$M\setminus\{a\}$ and it belongs to
$L^p(M)$ for any~$p\in [1, \, 2)$
(see e.g. Equation~(3.3) and Lemma~3.10 in~\cite{Ag-Reno}). 

\begin{proof}[Proof of Lemma~\ref{lemma:f}]
 We can write Equation~\eqref{canonicalvf} in terms of the Green function~$G$, as
 \[
  j(u^*[\a]) = -2\pi\sum_{k=1}^n d_k \star \d G(\cdot, \, a_k) 
  + \star\d\int_{M} \kappa(y) G(\cdot, \, y) \, \Vg(y) + \Xi(\a) 
 \]
 We take~$k\in\{1, \, \ldots, \, n\}$, $\a\in\UU$ and~$v\in T_{a_k} M$.
 By differentiating both sides of this equation with respect to~$a_k$,
 and recalling~\eqref{sigma_av}, we obtain
 \begin{equation} \label{nablaj-sigma}
  \D_{a_k} j(u^*[\a])[v]
  = -2\pi d_k \, \star\d \sigma(\cdot, \, a_k, \, v)
  + \D_{a_k}\Xi(\a)[g]
 \end{equation}
 In~\cite[Lemma~3.11, Step~1]{Ag-Reno} (see in particular Equation~(3.60)),
 we proved that there exists a smooth function
 $f(\cdot, \, a_k, \, v)\colon M\setminus\{a_k\}\to\R$ such 
 that\footnote{Strictly speaking, the arguments
 of~\cite[Lemma~3.11, Step~1]{Ag-Reno} only apply when~$M$
 is \emph{not} simply connected. However, similar arguments apply when~$M$ 
 is simply connected: since~$\sigma(\cdot, \, a_k, \, v)$
 is harmonic in~$M\setminus\{a_k\}$, it follows that 
 $\star\d\sigma(\cdot, \, a_k, \, v)$
 is smooth and closed in~$M\setminus\{a_k\}$. If~$M$ is a simply connected
 closed surface, then it is diffeomorphic to the sphere, so~$M\setminus\{a_k\}$
 is simply connected. Therefore, there exists a smooth 
 function~$f(\cdot, \, a_k, \, v)\colon M\setminus\{a_k\}\to\R$
 that satisfies~\eqref{nablaXi-sigma}, with~$\Xi=0$.}
 \begin{equation} \label{nablaXi-sigma}
  \D_{a_k}\Xi(\a)[v]
  = 2\pi d_k \, \star\d\sigma(\cdot, \, a_k, \, v)
  + 2\pi d_k \, \d f(\cdot, \, a_k, \, v)
 \end{equation}
 The function~$f(\cdot, \, a_k, \, v)$ is uniquely identified
 by~\eqref{nablaXi-sigma}, up to an additive constant. 
 We select a unique~$f(\cdot, \, a_k, \, v)$
 by imposing the constraint~\eqref{fzero}. Equation~\eqref{nablaj}
 now follows from~\eqref{nablaj-sigma} and~\eqref{nablaXi-sigma}.
 By taking the codifferential in both sides of~\eqref{nablaXi-sigma},
 we obtain
 \[
  -2\pi d_k \, \Delta f(\cdot, \, a_k, \, v) 
  = 2\pi d_k \, \d^*\d f(\cdot, \, a_k, \, v)
  = \d^*\left(\D_{a_k}\Xi(\a)[v]\right) = 0 
  \qquad \textrm{in } M\setminus\{a_j\},
 \]
 that is,~$f(\cdot, \, a_k, \, v)$ is harmonic in~$M\setminus\{a_k\}$.
 
 The proof of~\eqref{localdf} was given already
 in~\cite{Ag-Reno} (see Lemma~3.10 and Equation~(3.69)). 
 Therefore, it only remains to prove~\eqref{localf}
 (which, in particular, implies~$f(\cdot, \, a_k, \, v)\in L^p(M)$
 for any~$p<2$). 
 For simplicity of notation, we write~$f$ instead of~$f(\cdot, \, a_k, \, v)$.
 We work in normal geodesic coordinates centred at~$a_k$
 and we identify a geodesic ball $B_\delta(a_k)\subseteq M$
 with a Euclidean ball~$B_\delta(0)\subseteq\R^2$.
 (In particular, $a_k\in M$ is identified with the origin~$0\in\R^2$.)
 By Gauss' lemma, geodesic circles (centred at~$a_k$) in~$M$
 are mapped to Euclidean circles centered at the origin,
 geodesic rays from~$a_k$ are mapped to Euclidean rays from the origin,
 and~$\dist_g(x, \, a_k) = \abs{x}$
 for any~$x\in B_\delta(a_k)\simeq B_\delta(0)$.
 If~$w$, $\tilde{w}$ are 
 two smooth vector fields on~$B_\delta(0)$ such that
 $w(0) = \tilde{w}(0) = v$, then
 \[
  \frac{1}{2\pi\abs{x}} (i\nu(x), \, \tilde{w}(x))_g = 
  \frac{1}{2\pi\abs{x}} (i\nu(x), \, w(x))_g + \O(1)
 \]
 As a consequence, it suffices to prove~\eqref{localf}
 in case~$w$ is represented in coordinates by a constant map,
 $w(x) = v = (v^1, \, v^2)$ for any~$x\in B_\delta(0)$.
 
 We fix a reference point~$x_0 := (\rho_0, \, 0)\in B_\delta(0)$.
 For any~$x = (\rho\cos\phi, \, \rho\sin\phi)$ with~$0 < \rho < \rho_0$
 and~$0 \leq \phi < 2\pi$, let~$\gamma_{x,1}$ be the straght line segment
 of endpoints~$x_0$ and~$(\rho, \, 0)$, oriented from~$x_0$
 to~$(\rho, \, 0)$. Let~$\gamma_{x,2}$ be the circular arc,
 centred at the origin, from~$(\rho, \, 0)$ to~$x$,
 oriented in the anticlockwise direction. Recalling that
 the metric tensor in geodesic coordinates
 satisfies $g_{ij}(x) = \delta_{ij} + \O(\abs{x}^2)$,
 we obtain
 \[
  \begin{split}
   f(x) &= \int_{\gamma_{x,1}} \d f
    + \int_{\gamma_{x,2}} \d f + f(x_0) \\
   &\hspace{-.1cm}\stackrel{\eqref{localdf}}{=}
    - \int_\rho^{\rho_0} \left(\frac{v^2}{2\pi r^2} + \O(r^{-1})\right) \d r
    + \frac{1}{2\pi\rho}\int_0^{\phi} 
    \left(v^1 \cos t + v^2 \sin t\right) \d t + \O(1) + f(x_0) \\
   &= \frac{1}{2\pi\rho} \left(v^1 \sin\phi - v^2 \sin\phi\right) 
    + \O(\abs{\log\rho})
   = \frac{1}{2\pi\abs{x}} (i\nu(x), \, v)_g 
    + \O(\abs{\log\abs{x}}),
  \end{split}
 \]
 which completes the proof.
\end{proof}

\begin{lemma} \label{lemma:h}
 Let~$p_1$, \ldots, $p_n$ be distinct points in~$M$. 
 Let~$q>2$, $h\in L^q(M)$ such that $\int_M h \,\Vg = 0$.
 Then, there exists a
 $1$-form~$\tau\in W^{1,q}(M, \, T^*M)$ such that
 \begin{gather}
   \d^*\tau = h \label{tauh-l}\\
   \tau(p_k) = 0 \quad \textrm{for any } k\in\{1, \ldots, \, n\}
    \label{tauzero-l} \\
   \norm{\tau}_{W^{1,q}(M)} \leq C \norm{h}_{L^q(M)}\label{taunorm-l}
  \end{gather}
  for some constant~$C$ that depends only on~$M$, $q$ and the points~$p_k$.
\end{lemma}
\begin{proof}
 Since~$h$ has zero average by assumption, there exists 
 a unique function~$\alpha\in H^1(M)$ such that
 \[
  -\Delta\alpha = h, \qquad
  \int_M \alpha \, \Vg = 0
 \]
 By the Calderon-Zygmund theory (see \cite{gilb_trudi}),
 $\alpha\in W^{2,q}(M)$ and
 \begin{equation} \label{h1}
  \norm{\alpha}_{W^{2,q}(M)} \lesssim \norm{g}_{L^q(M)}
 \end{equation}
 Let~$B_1$, \ldots, $B_n$ be pairwise disjoint open balls,
 centred at the points~$p_1$, \ldots, $p_n$,
 with radii strictly smaller than the injectivity radius of~$M$.
 For each~$k$, let~$\eta_k\in C^\infty_{\mathrm{c}}(B_k)$
 be a cut-off function, such that~$\eta_k = 1$ in a neighbourhood 
 of~$a_k$. We define the function~$\varphi_k\colon B_k\to\R$ as
 \[
  \varphi_k(x) := -\eta_k(x) 
  (i\nabla\alpha(p_k), \, \mathrm{Exp}_{p_k}^{-1}(x))_g
 \]
 where~$\mathrm{Exp}_{p_k}\colon T_{p_k} M\to M$
 is the exponential map at~$p_k$. We have
 \[
  \star\d\varphi_k(p_k) = (i\nabla\varphi_k(p_k))^\flat
  = (\nabla\alpha(p_k))^\flat = \d\alpha(p_k)
 \]
 Moreover, the Sobolev 
 embedding~$W^{2,q}(M)\hookrightarrow C^1(M)$ implies
 \begin{equation} \label{h2}
  \norm{\d\varphi_k}_{W^{1,q}(B_k)}
  \lesssim \abs{\nabla\alpha(p_k)}
  \lesssim \norm{\alpha}_{W^{2,q}(M)}
  \stackrel{\eqref{h1}}{\lesssim} \norm{h}_{L^q(M)} 
 \end{equation}
 In this inequality, the multiplicative constants
 implied by the notation~$\lesssim$ depend on 
 the cut-off~$\eta_k$ and hence, on the radii of~$B_k$.
 However, the points~$p_k$ are fixed once and for all, 
 and so are the balls~$B_k$. 
 Therefore, the form
 \[
  \tau := \d\alpha - \sum_{k=1}^n \star\d\varphi_k
 \]
 satisfies all the desired properties.
\end{proof}

\begin{lemma} \label{lemma:uniquenessj}
 Let~$U\subseteq M$ be an open connected set. Let~$p\in U$.
 Let~$u$, $v$ be smooth fields on~$U$ such that
 \[
  \abs{u}= \abs{v} =1, \qquad j(u)=j(v) \quad \textrm{in } U,
  \qquad u(p) = v(p)
 \]
 Then, $u = v$ in~$U$.
\end{lemma}
\begin{proof}
 The lemma follows from~\cite[Lemma~6.4]{JerrardIgnat_full};
 we reproduce the argument, for convenience of the reader.
 Take an arbitrary point~$x\in U$. As~$U$ is open and connected,
 there exists a smooth curve~$\gamma\colon[0, \, 1]\to M$
 such that $\gamma(0) = p$, $\gamma(1) = x$. Let~$j:= j(u) = j(v)$.
 We consider the following Cauchy problem: 
 find a differentiable map~$w\colon [0, \, 1]\to TM$ such that
 \begin{equation} \label{Cauchy}
  \begin{cases}
   w(s)\in T_{\gamma(s)}M & \textrm{for any } s\in [0, \, 1]\\
   D_{\gamma^\prime(s)}w(s) = j(\gamma^\prime(s)) \, iw(\gamma(s)) &
    \textrm{for any } s\in [0, \, 1] \\
   w(0) = u(p)= v(p)
  \end{cases}
 \end{equation}
 Since~$j$ is smooth, the problem~\eqref{Cauchy} 
 has a unique solution. 
 By differentiating the constraint $\abs{u}^2 = \abs{v}^2 =1$, we deduce that
 $D_{\gamma^\prime(t)} u(\gamma(t))$ is parallel to~$iu(\gamma(t))$
 and $D_{\gamma^\prime(t)} v(\gamma(t))$ is parallel to~$iv(\gamma(t))$. 
 Using the definition of~$j(u)$, $j(v)$, we check that
 both~$t\mapsto u(\gamma(t))$ and~$t\mapsto v(\gamma(t))$
 are solutions to~\eqref{Cauchy}. Therefore, $u(\gamma(t)) = v(\gamma(t))$
 for any~$t\in [0, \, 1]$, and the lemma follows.
\end{proof}

\begin{lemma} \label{lemma:u*}
 For any~$p\in [1, \, 2)$, the map~$\UU\to L^p_{\tan}(M)$, 
 $\a\mapsto u^*[\a]$ is differentiable.
 For any~$k\in\{1, \, \ldots, \, n\}$, any~$\a\in\UU$ and
 any~$v\in T_{a_k}M$, there holds
 \[
  \D_{a_k}(u^*[\a])[v] = 2\pi d_k \, f(\cdot, \, a_k, \, v) \, i u^*[\a],
 \]
 where~$f(\cdot, \, a_k, \, v)$ is the function
 given by Lemma~\ref{lemma:f}.
\end{lemma}
\begin{proof}
 We fix a reference point~$\a^0 = (a_1^0, \, \ldots, \, a_n^0)\in\UU$.
 We split the proof into several steps.
 
 \setcounter{step}{0}
 \begin{step}
  First, we recall briefly the definition of~$\a\mapsto\Xi(\a)$.
  (We refer to~\cite[Section~5]{JerrardIgnat_full}, 
  \cite[Section~3.2]{Ag-Reno} for more details.)
  Let~$\gf$ denote the genus of~$M$; suppose for the moment that~$\gf>0$.
  We take smooth, closed, simple 
  curves~$\gamma_1$, \ldots, $\gamma_{2\gf}$ in~$M$
  whose homology classes generate the first homology group~$H_1(M; \, \Z)$.
  (Such curves exist; see e.g.~\cite[Lemma~5.2]{JerrardIgnat_full}).
  Up to a perturbation of the curves~$\gamma_h$, we may 
  assume without loss of generality that~$a_k^0$ does not
  belong to the image of~$\gamma_h$, for any~$k$ and~$h$.
  For any~$h$, we choose an orthogonal tangent
  frame~$\{\tau_{1,h}, \, \tau_{2,h}\}$ defined in a neighbourhood 
  of~$\gamma_h$, and consider the connection 
  $1$-form~$\mathcal{A}_k$ induced by~$\{\tau_{1,k}, \, \tau_{2,k}\}$ (see \cite[Section 5.2]{JerrardIgnat_full}.
  Now, let~$\a = (a_1, \, \ldots, \, a_n)$.
  We assume that each~$a_k$ is close enough to~$a_k^0$,
  so that~$a_k$ does not belong to the image of any of 
  the curves~$\gamma_h$, either. The map~$\Xi$ is defined
  in such a way that
  \begin{equation} \label{u*1,1}
   \int_{\gamma_h} \left( \Xi(\a) + \d^*\Phi(\a, \, \db) + \mathcal{A}_h\right) 
   \in 2\pi\Z \qquad \textrm{for any } h\in\{1, \, \ldots, \, 2\gf\}.
  \end{equation}
  (see \cite[Lemma~3.3]{Ag-Reno} --- in particular, Equation~(3.35)).
  As~$\a\to\a^0$, we have~$\Xi(\a)\to\Xi(\a^0)$
  (in any norm, since~$\Xi$ takes its values in
  the finite-dimensional space~$\Harm^1(M)$) and, by elliptic regularity,
  $\Phi(\a, \, \db)\to\Phi(\a^0, \, \db)$ locally uniformly 
  with all its derivatives, away from the points~$a_1^0$, \ldots, $a_n^0$.
  Therefore, the left-hand side of~\eqref{u*1,1} is continuous 
  as~$\a\to\a^0$. It follows that
  \begin{equation} \label{u*1,2}
   \int_{\gamma_h} \left(\Xi(\a) + \d^*\Phi(\a, \, \db) + \mathcal{A}_h\right) 
   = \int_{\gamma_h} \left(\Xi(\a^0) + \d^*\Phi(\a^0, \, \db) 
    + \mathcal{A}_h\right) 
  \end{equation}
  for any~$h\in\{1, \, \ldots, \, 2\gf\}$ and any~$\a$ 
  close enough to~$\a^0$. Equation~\eqref{u*1,2},
  together with~\eqref{canonicalvf}, implies
  \begin{equation} \label{u*1,3}
   \int_{\gamma_h} j(u^*[\a]) = \int_{\gamma_h} j(u^*[\a^0])
  \end{equation}
  for any~$h\in\{1, \, \ldots, \, 2\gf\}$.
  In case~$\gf = 0$, the condition~\eqref{u*1,3} is empty.
 \end{step}
 
 \begin{step}
  We fix~$\a = (a_1, \, \ldots, \, a_n)\in\UU$ with~$\a\neq\a^0$.
  We define $\mathbf{t} := (t_1, \, \ldots, \, t_n)\in\R^n$
  as~$t_k := \dist_g(a_k, \, a_k^0)$ for any~$k\in\{1, \, \ldots, \, n\}$,
  and
  \[
   M_\a := M\setminus \bigcup_{k=1}^n \bar{B}_{2t_k}(a_k^0)
  \]
  By construction, $a_k\notin M_\a$, for any~$k$.
  Equation~\eqref{canonicalvf} implies
  \[
   \d j(u^*[\a]) - \d j(u^*[\a^0])
   = 2\pi\sum_{k=1}^n d_k \left(\delta_{a_k} - \delta_{a_k^0}\right)
  \]
  and hence, $\d(j(u^*[\a]) - j(u^*[\a^0])) = 0$ in~$M_\a$.
  If~$\a$ is close enough to~$\a^0$, we have
  \begin{equation} \label{u*1,4}
   \begin{split}
   \int_{\partial B_{2t_k}(a_k^0)}\left(j(u^*[\a]) - j(u[\a^0])\right)
   &= 2\pi\ind(u^*[\a], \, \partial B_{2t_k}(a_k^0))
   - 2\pi\ind(u^*[\a^0], \, \partial B_{2t_k}(a_k^0)) \\
   &= 2\pi d_k - 2\pi d_k = 0 
   \end{split}
  \end{equation}
  Moreover, $j(u^*[\a]) - j(u^*[\a^0])$ integrates to zero
  on each of the curves~$\gamma_h$ considered above 
  --- see Equation~\eqref{u*1,3}.
  As the homology classes of~$\gamma_h$ generate~$H_1(M; \, \Z)$,
  from~\eqref{u*1,3} and~\eqref{u*1,4} we deduce 
  that~$j(u^*[\a]) - j(u^*[\a^0])$ is not only closed,
  but also exact in~$M_\a$. In other words, there exists
  a smooth function~$\beta_*[\a]\colon M_\a\to\R$ such that
  \begin{equation} \label{u*beta}
   j(u^*[\a]) - j(u^*[\a^0]) = \d\beta_*[\a] \qquad \textrm{in } M_\a
  \end{equation}
  (Equation~\eqref{u*beta} remains satisfied when~$M$ is simply connected,
  thanks to~\eqref{u*1,4}.) 
  The function~$\beta_*[\a]$ is uniquely identified by~\eqref{u*beta}
  up to an additive constant, and we select a unique~$\beta_*[\a]$
  by imposing 
  \begin{equation} \label{u*betazero}
   \beta_*[\a](b^0) = 0
  \end{equation}
  where~$b^0$ is the same point as in~\eqref{pointvalue}.
  The function~$\beta_*[\a]$ is harmonic in~$M_{\a}$:
  indeed, by taking the codifferential of~\eqref{u*beta},
  we obtain
  \begin{equation} \label{u*bharmonic}
   -\Delta\beta_*[\a] = \d^*\d\beta_*[\a]
   = \d^* j(u^*[\a]) - \d^* j(u^*[\a^0])
   \stackrel{\eqref{canonicalvf}}{=} 0
   \qquad \textrm{in } M_{\a}.
  \end{equation}
 \end{step}
 
 \begin{step}
  For any~$p\in [1, \, 2)$, $\a\mapsto j(u^*[\a])$ is continuous 
  as an~$L^p$-valued map, by Lemma~\ref{lemma:j}. Therefore, 
  the $L^p$-norm of~$\d\beta_*[\a]$ in~$M_\a$ is bounded independently
  of~$\a$, due to~\eqref{u*beta}.
  We extend~$\beta_*[\a]$ to a smooth function 
  $\beta\colon M\to\R$, in such a way that
  \begin{equation} \label{u*beta_norm}
   \norm{\d\beta[\a]}_{W^{1,p}(M)} \leq C_p
   \qquad \textrm{for any } p\in[1, \, 2)
  \end{equation}
  For instance, we may construct the extension inside~$B_{2t_k}(a_k^0)$ 
  in the following way: first, we define~$\beta_*[\a]$ 
  in~$B_{2t_k}(a_k^0)\setminus B_{t_k}(a_k^0)$ by
  reflection about the boundary of~$B_{2t_k}(a_k^0)$.
  Then, we take a cut-off function 
  $\eta\in C^\infty_{\mathrm{c}}(B_{2t_k}(a_k^0))$
  that is equal to~$1$ in a neighbourhood of~$\bar{B}_{t_k}(a_k^0)$
  and we define
  \[
   \beta[\a] := (1 - \eta) \beta_*[\a] + \beta_k\, \eta, 
   \qquad \textrm{where } \
   \beta_k := \fint_{B_{2t_k}(a_k^0)\setminus B_{t_k}(a_k^0)} \beta_*[\a] \, \Vg
  \]
  The Poincar\'e inequality implies
  \[
   \begin{split}
    \norm{\d\beta[\a]}_{L^p(B_{2t_k}(a_k^0)}
    &\lesssim \norm{\d\beta_*[\a]}_{L^p(B_{2t_k}(a_k^0)\setminus B_{t_k}(a_k^0))}
    + t_k^{-1}\norm{\beta_*[\a] - \beta_k}_{L^p(B_{2t_k}(a_k^0)
    \setminus B_{t_k}(a_k^0))} \\
    &\lesssim \norm{\d\beta_*[\a]}_{L^p(B_{2t_k}(a_k^0)\setminus B_{t_k}(a_k^0))}
    \lesssim \norm{\d\beta_*[\a]}_{L^p(B_{3t_k}(a_k^0)\setminus B_{2t_k}(a_k^0))}
   \end{split}
  \]
  and~\eqref{u*beta_norm} follows. We claim that
  \begin{equation} \label{u*beta_norm2}
   \norm{\beta(\a)}_{L^q(M)} \leq C_q 
   \qquad \textrm{for any } q\in [1, \, +\infty)
  \end{equation}
  Indeed, let~$K\subseteq M$ be a closed neighbourhood of~$b^0$
  that does not contain any of the points~$a_1^0$, \ldots, $a_n^0$.
  For~$\a$ close enough to~$\a^0$, we have $K\cap M_{\a}=\emptyset$
  --- in particular, $a_k\notin K$ for any~$k$.
  Then, by elliptic regularity, $j(u^*[\a])$ is bounded in~$L^\infty(K)$
  by a constant that does not depend on~$\a$. Taking~\eqref{u*beta}
  and~\eqref{u*betazero} into account, we deduce that
  \begin{equation} \label{u*beta_norm3}
   \norm{\beta[\a]}_{L^\infty(K)}
   = \norm{\beta_*[\a]}_{L^\infty(K)} \leq C
  \end{equation}
  for some constant~$C$ that does not depend on~$\a$.
  Equation~\eqref{u*beta_norm2} now follows from~\eqref{u*beta_norm}
  and~\eqref{u*beta_norm3}, via an `ad-hoc' Sobolev-Poincar\'e
  inequality. 
 \end{step}
 
 \begin{step}
  Let~$\mathrm{Exp}_a\colon T_aM\to M$
  denote the exponential map at a point~$a\in M$.
  If~$\a$ is close enough to~$\a^0$, then for any~$k\in\{1, \, \ldots, \, n\}$
  there exists a unique~$v_k\in T_{a_k^0} M$
  such that $\mathrm{Exp}_{a_k^0}(\abs{\mathbf{t}}v_k) = a_k$. 
  Let
  \[
   Q[\a] := \frac{\beta[\a]}{\abs{\mathbf{t}}} 
    - 2\pi\sum_{k=1}^n d_k f(\cdot, \, a_k^0, \, v_k)
  \]
  (where~$f(\cdot, \, a_k^0, \, v_k)$ is the function
  given by Lemma~\ref{lemma:f}). We claim that
  \begin{equation} \label{u*bdiff}
   \norm{Q[\a]}_{L^p(M)}\to 0 \quad \textrm{as } \a\to\a^0,
   \qquad \textrm{for any } p\in [1, \, 2).
  \end{equation}
  Once~\eqref{u*bdiff} is proved, observing that~$\beta[\a^0] = 0$,
  we will deduce that the map~$\UU\to L^p(M)$, $\a\mapsto\beta[\a]$
  is differentiable at~$\a^0$, with
  \begin{equation} \label{u*bdiff7}
   \D_{a_k}(\beta[\a^0])[v] = 2\pi d_k \, f(\cdot, \, a_k^0, \, v)
  \end{equation}
  for any~$k\in\{1, \, \ldots, \, n\}$ and any~$v\in T_{a_k^0} M$.

  We proceed to the proof of~\eqref{u*bdiff}.
  There is no loss of generality 
  in assuming~$p>1$, so we take~$p\in (1, \, 2)$.
  Let~$q\in (2, \, +\infty)$ be such that
  $1/p + 1/q = 1$. Let~$h\in L^q(M)$ be any function such that
  $\norm{h}_{L^q(M)}\leq 1$, $\int_M h\, \Vg = 0$. By Lemma~\ref{lemma:h},
  there exists a $1$-form $\tau\in W^{1,q}(M, \, T^*M)$
  such that
  \begin{gather}
   \d^*\tau = h \label{taug}\\
   \tau(a_k^0) = 0 \quad \textrm{for any } k\in\{1, \ldots, \, n\}
    \label{tauzero} \\
   \norm{\tau}_{W^{1,q}(M)} \leq C \label{taunorm}
  \end{gather}
  for some constant~$C$ that depends on~$M$, $q$ and~$\a^0$,
  but \emph{not} on~$\a$. Finally, let
  \begin{equation} \label{u*R}
   R[\a] := j(u^*[\a]) - j(u^*[\a^0]) - \d\beta[\a] 
  \end{equation}
  Since~$h$ has zero average, we have
  \[
   \int_{M} \left(Q[\a] - \fint_M Q[\a]\Vg \right) h \, \Vg 
    = \int_{M} Q[\a] \, h \, \Vg 
    = \int_{M} Q[\a] \, \d^*\tau \, \Vg 
  \]
  and hence, by integrating by parts,
  \begin{equation} \label{u*bdiff1}
   \begin{split}
    &\int_{M} \left(Q[\a] - \fint_M Q[\a]\Vg \right) h \, \Vg 
    = \int_{M} \left(\d Q[\a], \, \tau\right)_g \Vg \\
    &\hspace{1cm} \stackrel{\eqref{u*R}}{=} 
    \int_{M} \left(\frac{j(u^*[\a]) - j(u^*[\a^0])}{\abs{\mathbf{t}}} 
     - 2\pi\sum_{k=1}^n d_k \, \d f(\cdot, \, a_k^0, \, v_k), \,
     \tau\right)_g \Vg 
    + \int_{M} \left(\frac{R[\a]}{\abs{\mathbf{t}}}, \, \tau\right)_g \Vg \\
    &\hspace{1cm} =: I_1 + I_2
   \end{split}
  \end{equation}
  The map~$a\mapsto j(u^*[\a])$ is differentiable as a~$W^{-1,p}$-valued
  map, and its differential is given in terms of~$f(\cdot, \, a_k^0, \, v)$
  by Lemma~\ref{lemma:f}. Therefore, keeping~\eqref{taunorm}
  into account, we deduce
  \begin{equation} \label{u*bdiff3}
   I_1 \lesssim \norm{\frac{j(u^*[\a]) - j(u^*[\a^0])}{\abs{\mathbf{t}}} 
     - 2\pi\sum_{k=1}^n d_k \, \d f(\cdot, \, a_k^0, \, v_k)}_{W^{-1,p}(M)}
     \to 0 \qquad \textrm{as } \a\to\a^0.
  \end{equation}
  It remains to estimate the other term, $I_2$.
  As~$q>2$, the $1$-form~$\tau$ is H\"older continuous of
  exponent~$1 - 2/q$, by Sobolev embedding. Then,
  Equation~\eqref{tauzero} implies
  \[
   \begin{split}
    \norm{\tau}_{L^q(M\setminus M_\a)}
    \lesssim \norm{\tau}_{W^{1,q}(M)}
    \sum_{k=1}^n t_k^{1 - 2/q} \, \mathrm{Vol}(B_{2t_k}(a_k^0))^{1/q}
    \lesssim \norm{\tau}_{W^{1,q}(M)}
    \sum_{k=1}^n t_k \stackrel{\eqref{taunorm}}{\lesssim} \abs{\mathbf{t}}
   \end{split}
  \] 
  On the other hand, the $1$-form~$R[\a]$ is supported
  in~$M\setminus M_\a$ and belongs to~$L^s(M)$ for any~$s<2$,
  due to~\eqref{u*beta} and~\eqref{u*beta_norm}.
  In fact, the~$L^s$-norm of~$R[\a]$ is bounded independently of~$\a$,
  thanks to Lemma~\ref{lemma:j} and~\eqref{u*beta_norm}. Therefore,
  \begin{equation} \label{u*bdiff4}
   \begin{split}
    I_2 \lesssim \frac{1}{\abs{\mathbf{t}}} 
    \norm{R[\a]}_{L^p(M\setminus M_\a)} \norm{\tau}_{L^q(M\setminus M_\a)}
    \lesssim \norm{R[\a]}_{L^p(M\setminus M_\a)}
    \lesssim C_s
    \abs{\mathbf{t}}^{2/p- 2/s}
   \end{split}
  \end{equation}
  for any~$s$ such that~$p < s < 2$. Combining~\eqref{u*bdiff1},
  \eqref{u*bdiff3} and~\eqref{u*bdiff4}, and taking the supremum 
  over all admissible functions~$h$, we deduce
  \begin{equation} \label{u*bdiff5}
   \norm{Q[\a] - \fint_{M} Q[\a]\, \Vg}_{L^p(M)}
     \to 0 \qquad \textrm{as } \a\to\a^0.
  \end{equation}
  Let~$V\subseteq M$ be an open neighbourhood of the point~$b^0$,
  such that~$a_k^0\notin\overline{V}$ for any~$k$.
  For~$\a$ is close enough to~$\a^0$, we have
  $V\subseteq M_{\a}$. Then, the function~$Q[\a]$
  is harmonic in~$V$, because both~$\beta[\a]$
  and~$f(\cdot, \, d_k, \, v_k)$ are harmonic in~$M_{\a}$
  --- by Equation~\eqref{u*bharmonic} and Lemma~\ref{lemma:f}, respectively.
  Equation~\eqref{u*bdiff5} and the Calderon-Zygmund
  theoryimply that $Q[\a] - \fint_M Q[\a] \, \Vg\to 0$
  in~$W^{2,p}_{\mathrm{loc}}(V)$ as~$\a\to\a^0$.
  By Sobolev embedding, it follows that $Q[\a] - \fint_M Q[\a] \, \Vg\to 0$
  locally uniformly in~$V$ as~$\a\to\a^0$. However, we have
  \[
   Q[\a](b^0) \stackrel{\eqref{fzero}, \, \eqref{u*betazero}}{=} 0
  \]
  which implies
  \begin{equation} \label{u*bdiff6}
   \fint_{M} Q[\a]\, \Vg \to 0 \qquad \textrm{as }\a\to\a^0.
  \end{equation}
  Combining~\eqref{u*bdiff5} with~\eqref{u*bdiff6}, we obtain~\eqref{u*bdiff}.
 \end{step}

 \begin{step}
  We consider the vector field~$w_*[\a] := e^{i\beta[\a]}u^*[\a^0]$.
  Due to~\eqref{u*bdiff} and~\eqref{u*bdiff7},
  the map~$\UU\to L^p_{\tan}(M)$, $\a\mapsto w_*[\a]$ is differentiable
  at~$\a^0$, with
  \begin{equation} \label{u*gradv}
   \D_{a_k}(w_*[\a^0]) [v] = 2\pi d_k \, f(\cdot, \, a_k^0, \, v) \, iu^*[\a^0]
  \end{equation}
  for any~$k\in\{1, \, \ldots, \, n\}$ and any~$v\in T_{a_k^0}M$.
  By an explicit computation, we see
  \begin{equation} \label{u*v}
   j(w_*[\a]) = \d\beta[\a] + j(u^*[\a^0])
   \stackrel{\eqref{u*beta}}{=} j(u^*[\a])
   \qquad \textrm{in } M_{\a}
  \end{equation}
  and
  \begin{equation} \label{u*vzero}
   w_*[\a](b^0) \stackrel{\eqref{u*betazero}}{=} u^*[\a](b^0)
  \end{equation}
  By Lemma~\ref{lemma:uniquenessj}, we deduce
  that~$w_*[\a] = u^*[\a]$ in~$M_{\a}$. For any~$p\in [1, \, 2)$,
  it follows that
  \begin{equation} \label{u*gradv2}
   \begin{split}
    \frac{\norm{w_*[\a] - u^*[\a]}_{L^p(M)}}{\abs{\mathbf{t}}}
    = \frac{\norm{w_*[\a] - u^*[\a]}_{L^p(M\setminus M_{\a})}}{\abs{\mathbf{t}}}
    \leq \frac{2\mathrm{Vol}(M\setminus M_{\a})^{1/p}}{\abs{\mathbf{t}}}
    \lesssim \abs{\mathbf{t}}^{2/p-1}\to 0
   \end{split}
  \end{equation}
  as~$\a\to 0$.
  The lemma follows from~\eqref{u*gradv} and~\eqref{u*gradv2}.
  \qedhere
 \end{step}
\end{proof}

\begin{proof}[Proof of Proposition~\ref{prop:Gdiff}]
 Let~$\theta\in H^1(M)$. Let~$\a_0 = (a_1^0, \, \ldots, \, a_n^0)\in\UU$,
 $\a=(a_1, \, \ldots, \, a_n)\in\UU$. As in the proof of Lemma~\ref{lemma:u*},
 we define~$\mathbf{t} := (t_1, \, \ldots, \, t_k)\in\R^n$
 by~$t_k := \dist_g(a_k, \, a_k^0)$. For any index~$k$,
 we let~$v_k$ be the unique element of~$T_{a_k^0} M$
 such that $\mathrm{Exp}_{a_k^0}(\abs{\mathbf{t}}v_k) = a_k$.
 Let~$p\in (1, \, 2)$. By Lemma~\ref{lemma:u*}, we can write
 \[
  u^*[\a] = u^*[\a^0] + \abs{\mathbf{t}} z[\a] + \abs{\mathbf{t}} r[\a],
 \]
 where
 \begin{gather}
  z[\a] := 2\pi\sum_{k=1}^n 
   d_k \, f(\cdot, \, a_k^0, \, v_k) \, i u^*[\a^0] \\
  \norm{r[\a]}_{L^p(M)} \to 0 \quad \textrm{as } \a\to\a^0
   \label{Gdiffr}
 \end{gather}
 From the definition of the functional~$\GG$, Equation~\ref{G}, we deduce
 \[
  \begin{split}
   &\GG(u^*[\a], \, \theta) - \GG(u^*[\a^0], \, \theta)
   = \frac{1}{2}\int_M \left(\abs{\Sh(e^{i\theta}u^*[\a])}^2_g
    - \abs{\Sh(e^{i\theta} u^*[\a^0])}^2_g\right)\Vg \\
   &\quad = \abs{\mathbf{t}}\int_M \left(\Sh(e^{i\theta}u^*[\a^0]), \, 
    \Sh(e^{i\theta}z[\a] + e^{i\theta}r[\a])\right)_g \Vg 
    + \frac{\abs{\mathbf{t}}^2}{2}
    \int_M \abs{\Sh(e^{i\theta}z[\a] + e^{i\theta}r[\a])}^2_g \Vg \\
   &\quad = \abs{\mathbf{t}}\int_M \left(\Sh(e^{i\theta}u^*[\a^0]), \, 
    \Sh(e^{i\theta}z[\a] + e^{i\theta}r[\a])\right)_g \Vg \\
   &\qquad\qquad\qquad + \frac{\abs{\mathbf{t}}}{2}
    \int_M \left(\Sh(e^{i\theta} u^*[\a] - e^{i\theta}u^*[\a^0]), \,
    \Sh(e^{i\theta}z[\a] + e^{i\theta}r[\a])\right)_g \Vg  
  \end{split}
 \]
 Therefore, in order to prove that~$\a\mapsto\GG(u^*[\a], \, \theta)$
 is differentiable at~$\a^0$, it suffices to show that
 \begin{gather} 
  \int_M \left(\Sh(e^{i\theta}u^*[\a^0]), \, 
    \Sh(e^{i\theta}r[\a])\right)_g \Vg \to 0, \label{Gdiff1} \\
  \int_M \left(\Sh(e^{i\theta} u^*[\a] - e^{i\theta}u^*[\a^0]), \,
    \Sh(e^{i\theta}z[\a] + e^{i\theta}r[\a])\right)_g \to 0  \label{Gdiff2}
 \end{gather}
 as~$\a\to\a^0$. Equation~\eqref{Gdiff1}
 follows immediately from~\eqref{Gdiffr}, taking into account that
 $u^*[\a^0]$ is bounded. 
 Lemma~\ref{lemma:u*} implies that $u^*[\a]\to u^*[\a^0]$ 
 strongly in~$L^q_{\tan}(M)$ for any~$q < 2$, as~$\a\to\a^0$.
 Moreover, $\norm{u^*[\a]}_{L^\infty(M)} = 1$ for any~$\a$
 and hence, by interpolation, $u^*[\a]\to u^*[\a^0]$ in~$L^q_{\tan}(M)$ 
 for any~$q < +\infty$. Choosing~$q$ in such a way that
 $1/p + 1/q = 1$, we deduce~\eqref{Gdiff2}
 from the H\"older inequality. This proves that
 $\a\mapsto\GG(u^*[\a], \, \theta)$ is differentiable 
 at~$\a^0$.
 
 As a byproduct of the argument above, we obtain
 a characterisation of the differential of~$\GG$:
 \begin{equation} \label{Gdiff-f}
  \D_{a_k}\GG(u^*[\a^0], \, \theta) [v]
  = 2\pi d_k \int_M f(\cdot, \, a_k^0, \, v)
   \left(\Sh(e^{i\theta}u^*[\a^0]), \, 
    \Sh(e^{i\theta} iu^*[\a^0])\right)_g \Vg
 \end{equation}
 for any index~$k$ and any~$v\in T_{a_k^0} M$.
 We work at fixed~$k$ and~$v$ from now on. To simplify the 
 notation, we write~$f$ instead of~$f(\cdot, \, a_k^0, \, v)$.
 We assume that~$\theta$ is a critical point for~$\GG(u^*[\a^0], \, \cdot)$
 --- that is, $\theta$ satisfies the Euler-Lagrange equation
 \begin{equation} \label{Gdiff-EL}
  -\Delta\theta + \left(\Sh(e^{i\theta}u^*[\a^0]), \, 
    \Sh(e^{i\theta} iu^*[\a^0])\right)_g = 0
 \end{equation}
 Equation~\eqref{Gdiff-EL} implies that~$\Delta\theta\in L^\infty(M)$.
 By elliptic regularity, it follows that
 $\theta\in W^{2,q}(M)$ for any~$q<+\infty$
 and, by Sobolev embedding, $\theta\in C^{1,\alpha}(M)$
 for any~$\alpha\in(0, \, 1)$.
 Equation~\eqref{Gdiff-f} may be written equivalently as
 \begin{equation} \label{Gdiff3}
  \D_{a_k}\GG(u^*[\a^0], \, \theta) [v]
  = 2\pi d_k \int_M f \, \Delta\theta\, \Vg
 \end{equation}
 We integrate by parts (and recall that~$-\Delta f = 0$ 
 in~$M\setminus\{a_k^0\}$, by Lemma~\ref{lemma:f}):
 \[
  \begin{split}
   \frac{\D_{a_k}\GG(u^*[\a^0], \, \theta) [v]}{2\pi d_k}
   &= \lim_{\eta\to 0}\int_{M\setminus B_\eta(a_k^0)} 
    f \, \Delta(\theta - \theta(a_k^0)) \, \Vg \\
   &= \lim_{\eta\to 0}\int_{\partial B_\eta(a_k^0)} 
    \left(- f\frac{\partial\theta}{\partial\nu}
   + \left(\theta - \theta(a_k^0)\right)
    \frac{\partial f}{\partial\nu}\right) \d\H^1
  \end{split}
 \]
 where~$\nu$ is the outward unit normal to~$\partial B_\eta(a_k^0)$.
 Let~$w$ be a smooth field, defined in a neighborhood of~$a_k^0$,
 such that~$w(a_k^0) = w$.
 By applying Lemma~\ref{lemma:f}, we obtain
 \[
  \begin{split}
   \frac{\D_{a_k}\GG(u^*[\a^0], \, \theta) [v]}{2\pi d_k}
   &= \lim_{\eta\to 0} \frac{1}{2\pi\eta} \int_{\partial B_\eta(a_k^0)} 
   (i\nu, \, w)_g \left(\frac{\partial\theta}{\partial\nu}
   + \frac{\theta - \theta(a_k^0)}{\eta}\right) \d\H^1
  \end{split}
 \]
 As we have already seen, elliptic regularity implies 
 that~$\theta\in C^1(M)$. Therefore,
 \begin{equation} \label{Gdiff4}
  \begin{split}
   \frac{\D_{a_k}\GG(u^*[\a^0], \, \theta) [v]}{2\pi d_k}
   &= \lim_{\eta\to 0} \frac{1}{\pi\eta} \int_{\partial B_\eta(a_k^0)} 
    (i\nu, \, w)_g \, \frac{\partial\theta}{\partial\nu} \, \d\H^1
  \end{split}
 \end{equation}
 To evaluate the integral at the right-hand side of~\eqref{Gdiff4},
 we work in geodesic normal coordinates centred at~$a_k^0$,
 and identify points in~$B_\eta(a_k^0)\subseteq M$
 with points in~$B_\eta(0)\subseteq\R^2$, with~$a_k^0=0$. We choose as~$w$
 the vector field that is represented in coordinates
 by a constant map, $w = v$. The geodesic coordinates
 map the geodesic circle~$\partial B_\eta(a_k^0)\subseteq M$ 
 to~$\partial B_\eta(0)\subseteq\R^2$,
 and the metric tensor in geodesic coordinates satisfies
 $g_{ij}(x) = \delta_{ij} + \O(\abs{x}^2)$. Therefore, we obtain
 \begin{equation*} 
  \begin{split}
   \frac{\D_{a_k}\GG(u^*[\a^0], \, \theta) [v]}{2\pi d_k}
   &= \lim_{\eta\to 0} \frac{1}{\pi} \int_0^{2\pi} 
    \left(-v^1\sin\phi + v^2\cos\phi\right) \,
    \left(\partial_1\theta(\eta e^{i\phi}) \cos\phi 
    + \partial_2\theta(\eta e^{i\phi}) \sin\phi\right) \d\phi \\
   &= v^2\,\partial_1\theta(0) - v^1\,\partial_2\theta(0)
   = (i\nabla\theta(0), \, v)_g
  \end{split}
 \end{equation*}
 and the proposition follows.
\end{proof}

\subsection{ The gradient of the renormalized energy\texorpdfstring{~$W$}{ W}}

We turn our attention to the renormalised energy~$W$,
defined in~\eqref{eq:reno_intro}. We recall the definition here.
Let~$\mathscr{U}\subseteq M^n$ be as in Section~\ref{sect:G}.
For any~$\a\in\mathscr{U}$, any~$\db\in\Z^n$ 
with~$(\a, \, \db)\in\mathscr{A}^n$ and any~$\xi\in\mathcal{L}(\a, \,\db)$,
we consider the unique ``reference'' canonical harmonic field
$u^*[\a, \, \db, \, \xi]$, defined by the 
conditions~\eqref{canonicalvf} and~\eqref{pointvalue}. 
The renormalized energy is given as 
\begin{equation}
\label{eq:reno}
  W(\a, \, \db, \, \xi, \, \theta) := W^{\intr}(\a, \, \db, \, \xi) 
  + \GG(u^*[\a, \, \db, \, \xi], \, \theta),
\end{equation}
where~$W^{\intr}$ is the intrinsic component, 
defined in~\eqref{eq:renoJIintr}, and~$\GG$ is as in~\eqref{G}. 
We define the gradient of~$W$
with respect to the variable~$\a$ as in~\cite{Ag-Reno}:
for any~$(\a, \, \db)\in\mathscr{A}^n$ 
and~$\xi\in\mathcal{L}(\a, \, \db)$, there exists a unique
smooth map~$b\in M^n\mapsto\Xi(\b)\in\Harm^1(M)$, 
locally defined for~$\b$ in a neighbourhood of~$\a$,
such that $\Xi(\b)\in\mathcal{L}(\a, \, \db)$ for any~$\b$
and~$\Xi(\a) = \xi$ \cite[Lemma~3.3]{Ag-Reno}. Then, we define
\begin{equation} \label{gradW}
 \nabla_{\a} W(\a, \, \db, \, \xi, \, \theta)
  := \nabla_{\b} W(\b, \, \, \db, \, \Xi(\b), \, \theta)_{|\b=\a}
\end{equation}
The right-hand side of~\eqref{gradW} is well-defined,
because~$W^\intr$ is smooth~\cite[Proposition~2.4]{JerrardIgnat_full}
and~$\GG$ is differentiable (by Proposition~\ref{prop:Gdiff}).
Moreover, if~$\theta$ is a critical point 
for~$\GG(u^*[\a, \, \d, \, \xi], \, \cdot)$, then
the gradient~$\nabla_{\a} W$ can be further characterized as
\begin{equation} \label{gradWbis}
 \nabla_{\a} W(\a, \, \db, \, \xi, \, \theta)
  = \nabla_{\a} W^\intr(\a, \, \db, \, \xi)
  + 2\pi \sum_{j=1}^n d_j \, i \, \nabla\theta(a_j)
\end{equation}
For further properties of~$\nabla_{\a} W^{\intr}$,
see~\cite[Section~3]{Ag-Reno}.


The next Proposition is the analogue of \cite[Proposition 3.4]{Ag-Reno} (see also \cite[Theorem VII.4, Theorem VIII.3]{BBH} and \cite[Theorem 5.1]{lin1} for the Euclidean versions) in the extrinsic framework and characterize~$\nabla_\a W$ in terms of the extrinsic analogue of the
harmonic canonical vector field, $u_0= e^{i\theta}u^*$. 
Given $(\a,\db)$ we consider geodesic balls $B_\eta(a_j)$ centered in $a_j$ and with radius $\eta$ such that their closures are pairwise disjoint. 
In what follows we denote 
with~$\nu$ the exterior unit normal to $\partial B_\eta(a_j)$.

\begin{prop} \label{prop:gradient_reno}
 Let~$(\a, \, \db)\in\mathscr{A}^n$, with~$\a\in\mathscr{U}$,
 and~$\xi\in\mathcal{L}(\a, \, \db)$.
 Let~$u^* = u^*[\a, \, \db, \, \xi]$ be the ``reference''
 canonical harmonic field defined in~\eqref{canonicalvf}, \eqref{pointvalue}.
 Let~$u_0:= e^{i\theta}u^*$ with 
 $\theta\in C^{1}(M)$ a critical point of $\GG(u^*, \, \cdot)$.
 Let~$j\in\{1, \, \ldots, \, n\}$ be fixed, and 
 let~$e$ be a smooth vector field defined in a neighbourhood~$U$
 of~$a_j$, such that
 \begin{equation*} 
  \div e = 0 \qquad \textrm{in } U.
 \end{equation*}
 Then, we have
 \begin{equation*} 
  \begin{split}
   \lim_{\eta\to 0} \int_{\partial B_\eta(a_j)} 
    \left( (D_{e} u_0, \, D_{\nu} u_0)_g
    - \frac{1}{2} \abs{D u_{0}}^2_g (\nu, \, e)_g \right)\d\H^1
   = \left(\nabla_{a_j} W(\a, \, \db, \, \xi, \,\theta), \, e(a_j)\right)_g
  \end{split}
 \end{equation*}
\end{prop}
\begin{proof}
Let us set 
\[
I_\eta(u_0):=\int_{\partial B_\eta(a_j)} 
    \left( (D_{e} u_0, \, D_{\nu} u_0)_g
    - \frac{1}{2} \abs{D u_{0}}^2_g (\nu, \, e)_g \right)\d\H^1.
\]
 We define~$I_\eta(u^*)$ in a completely analogous way.
A standard computation (see \cite{Ag-Reno}) shows that 
\[
I_\eta(u_0) = \int_{\partial B_\eta(a_j)} \left( j(u_0)[e] \, j(u_0)[\nu]
    - \frac{1}{2} \abs{j(u_0)}^2_g (\nu, \, e)_g \right)\d\H^1,
\]
with (recall that $u_0= e^{i\theta}u^*$)
\[
j(u_0) = \d \theta + \d \Psi +\xi,
\]
where $\Psi$ is given in \eqref{Phiad} and $\xi\in \mathcal{L}(\a,\db)$. 
In \cite[Proposition 3.4]{Ag-Reno} we proved that 
\begin{equation}
\label{eq:I(u*)}
 \lim_{\eta\to 0}I_\eta(u^*) 
 = \left(\nabla_{a_j}W^{\intr}(\a, \, \db, \, \xi), \, e\right),
\end{equation}
and that (see \cite[Formula 3.71]{Ag-Reno})
\begin{equation}
\label{eq:O1j}
j(u^*) = \d^* \Psi +\xi = \frac{d_j}{\eta}(i\nu)^\flat + \O(1), 
\end{equation}
where we used that $\xi$ is smooth.
Now, we write $I_\eta(u_0)$ in terms of $I_\eta(u^*)$. 
There holds
\begin{equation*}
\begin{split}
I_\eta(u_0) = I_\eta(u^*) & 
+\int_{\partial B_\eta(a_j)}\left( j(u^*)[\nu] \, \d \theta[e] 
+ \d\theta[\nu] \, \d \theta[e]\right. \\
&\qquad \left. -(j(u^*), \, \d\theta) \, (\nu, \, e)\right)\d\H^1 + \O(\eta).
\end{split}
\end{equation*}
Therefore, since \eqref{eq:O1j} implies that $j(u^*)[\nu]= \O_{\eta\to 0}(1)$, we obtain 
\begin{equation}
\begin{split}
I_\eta(u_0) &= I_\eta(u^*) + \frac{d_j}{\eta}\int_{\partial B_\eta(a_j)}\left(\partial_\nu\theta (i\nu,e)_g -(i\nu,\nabla\theta)_g(\nu,e)_g\right)\d\H^1 + \O(\eta)\\
& = I_\eta(u^*)  + 2\pi d_g\fint_{\partial B_\eta(a_j)}\left(-(\nabla\theta,\nu)_g(\nu,ie)_g
 -(\nabla\theta,i\nu)_g(i\nu,ie)_g\right)\d\H^1 + \O(\eta)\\
&= I_\eta(u^*) -2\pi d_j\fint_{\partial B_\eta(a_j)}(\nabla \theta, \, i e)\d\H^1 + \O(\eta)\\
&= I_\eta(u^*) + 2\pi d_j\fint_{\partial B_\eta(a_j)}(i\nabla \theta, \, e)\d\H^1 + \O(\eta).
\end{split}
\end{equation}
Recalling \eqref{eq:I(u*)} and Proposition \ref{prop:Gdiff}, we conclude that 
\[
\lim_{\eta\to 0}I_\eta(u_0) = \left(\nabla_{a_j}W^{\intr}(\a,\db,\xi) + 2\pi d_j i\nabla \theta(a_j), \, e\right)_g = \left(\nabla_{a_j} W(\a, \, \db, \, \xi, \,\theta), \, e(a_j)\right)_g.
\qedhere
\]
\end{proof}

\section{Vortex dynamics}
\label{sec:dyna_vortex}

\subsection{Analysis of a non-homogeneous stationary Ginzburg-Landau equation}
\label{sect:inhomGL}

In this section, we consider solutions~$u_\eps\in H^1_{\tang}(M)$
of the non-homogeneous equation
\begin{equation} \label{inhomGL}
 -\Delta_g u_\eps + \frac{1}{\eps^2} \left(\abs{u_\eps}^2_g - 1\right) u_\eps +\Sh^2 u_\eps
  = f_\eps, 
\end{equation}
where~$(f_\eps)_{\eps>0}$ is a given sequence of vector fields on~$M$,
which we assume to be bounded in~$L^2_{\tang}(M)$ (at least).
For the time being, there is still no dependence on time
in equation~\eqref{inhomGL}.
However, later in this Section we will apply these results
to the parabolic Ginzburg-Landau equation~\eqref{eq:GL_intro},
by working at fixed time~$t$ and taking~$f_\eps = -\partial_t u_\eps(t)/\abs{\log\eps}$.

We assume that~$u_\eps$ satisfies
\begin{equation} \label{hp:H1bounds-flat-bis} 
 \begin{split}
  \omega(u_\eps) \to 2\pi\sum_{j=1}^n d_j \delta_{a_j} \qquad 
  \textrm{in } W^{-1,1}(M) \quad  \textrm{as } \eps\to 0,
 \end{split} 
\end{equation}
where~$a_1$, \ldots, $a_n$ are distinct points in $M$ and $d_1$, \ldots, $d_n$
are non-zero integers with~$\sum_{j=1}^n d_j = \chi(M)$.
As usual, we denote~$\a := (a_1, \, \ldots, \, a_n)$
and~$\db := (d_1, \, \ldots, \, d_n)$. 
We further assume that
\begin{gather}
 \norm{u_\eps}_{L^\infty(M)} \leq C_0 \label{hp:H1bounds-Linfty-bis} \\
 F_{\eps}^{\extr}(u_\eps)=\int_{M} \left(\frac{1}{2}\abs{D u_\eps}^2_g 
  + \frac{1}{4\eps^2}(1 -  \abs{u_\eps}^2_g)^2 + \abs{\Sh u_\eps}^2\right)\Vg 
  \le \pi n \abs{\log\eps}  + C_0 \label{hp:H1bounds-energy-bis}
\end{gather}
for some $\eps$-independent constant~$C_0$.
Since 
\[
F_\eps^{\intr}(u_\eps)\le F_{\eps}^{\extr}(u_\eps),
\]
we may apply~\cite[Lemma 4.3, Corollary 4.5, Lemma~4.9]{Ag-Reno}
and find a limit field~$u_0$ such that~$\abs{u_0}_g = 1$ 
a.e.~in~$M$ and
\begin{equation} \label{convu0}
 u_\eps\xrightarrow{\eps\to 0} u_0 \qquad \textrm{strongly in } W^{1,p}(M)
 \textrm{ and in } H^1_{\mathrm{loc}}(M\setminus\{a_1, \, \ldots, \, a_n\})
\end{equation}
up to extraction of a subsequence.
Note that~\eqref{hp:H1bounds-Linfty-bis} and~\eqref{convu0} imply
$u_\eps\xrightarrow{\eps\to 0}u_0$ strongly in $L^q(M)$ for any~$q<+\infty$. 

\begin{prop}
\label{prop:gammaliminf_extr}
 Let~$u_\eps\in H^1_{\tang}(M)$ be a sequence of solutions of~\eqref{inhomGL}
 that satisfies~\eqref{hp:H1bounds-flat-bis}, 
 \eqref{hp:H1bounds-Linfty-bis}, \eqref{hp:H1bounds-energy-bis}.
 Assume that the sequence~$(f_\eps)_{\eps>0}$ 
 converges to $0$ in~$L^2_{\tang}(M)$.
 Then there exists $\xi\in \Harm^1(M)$ with $\xi\in \mathcal{L}(\a,\db)$
 and $\theta\in C^1(M)$ such that 
 \begin{equation}
 \label{eq:hodgeu0}
 j(u_0) = \d \theta + \d^*\Psi + \xi,
 \end{equation}
 where $\Psi$ is defined in \eqref{Phiad}.
 Moreover, $u_0$ can be written as
 $u_0 = e^{i\theta} u^*[\a, \, \db, \, \xi]$,
 where~$u^*[\a, \, \db, \, \xi]$ is the ``reference'' 
 canonical harmonic field defined in~\eqref{canonicalvf}--\eqref{pointvalue},
 and~$\theta$ is a critical point 
 of the functional~$\GG(u^*[\a, \, \db, \, \xi], \, \cdot)$. Finally,
 \begin{equation}
 \label{eq:gammalinf_extr}
  \liminf_{\eps\to 0}\left(F_\eps^{\extr}(u_\eps) - n\pi\abs{\log\eps} -n\gamma\right)
  \ge W(\a, \, \db, \, \xi, \, \theta).
 \end{equation}
\end{prop}

Before we proceed to the proof of Proposition~\ref{prop:gammaliminf_extr},
we state an auxiliary result. For any~$\eps> 0$, we consider
the Hodge decomposition of the $1$-form~$j(u_\eps)$ ---
that is, we write
\begin{equation} \label{Hodge-j}
 j(u_\eps) = \d \theta_\eps + \d^* \beta_\eps +\xi_\eps.
\end{equation}
where~$\theta_\eps$ is a smooth function,
$\beta_\eps$ is a $2$-form and~$\xi_\eps := \P_H(j(u_\eps))$
is a harmonic~$1$-form. The function~$\theta_\eps$
is uniquely determined up to an additive constant.
We impose the renormalisation condition
\begin{equation} \label{Hodge-j-average}
 \int_M \theta_\eps \, \Vg = 0
\end{equation}
We denote by~$\Lambda$ the set of cluster points of~$(\theta_{\eps})_{\eps> 0}$
in~$H^1(M)$. By definition, a function~$\theta\in H^1(M)$ belongs to~$\Lambda$
if and only if there exists a subsequence~$(\theta_{\eps_j})_{j\in\N}$
such that~$\theta_{\eps_j}\xrightarrow{j\to+\infty}\theta$ in~$H^1(M)$.

\begin{lemma} \label{lemma:thetaeps}
 Let~$u_\eps\in H^1_{\tang}(M)$ be a sequence of solutions of~\eqref{inhomGL}
 that satisfies~\eqref{hp:H1bounds-flat-bis}, 
 \eqref{hp:H1bounds-Linfty-bis}, \eqref{hp:H1bounds-energy-bis}. Assume
 that~$f_\eps\xrightarrow{\eps\to 0} 0$ in~$L^2_{\tang}(M)$.
 Then, the following statements hold:
 \begin{enumerate}[label=(\roman*)]
  \item the sequence~$(\theta_\eps)_{\eps> 0}$
  defined by~\eqref{Hodge-j}--\eqref{Hodge-j-average}
  is relatively compact in~$H^1(M)$
  (in particular, $\Lambda$ is nonempty);
  \item $\Lambda$ is a closed subset of~$C^1(M)$;
  \item there exists a compact set~$K_*\subseteq C^1(M)$
  (depending only on~$M$ and the constant~$C_0$ in~\eqref{hp:H1bounds-Linfty-bis})
  such that $\Lambda\subseteq K_*$.
 \end{enumerate}
\end{lemma}
\begin{proof}
 We claim that~$\theta_\eps$ satisfies an elliptic equation;
 the lemma will follow by elliptic regularity theory.
 Let~$\psi\colon M\to\R$ be a smooth test function.
 We test Equation~\eqref{inhomGL} aginst~$i u_\eps \psi$:
 \[
  \int_{M}\left(-\Delta_g u_{\eps}, iu_{\eps}\right)_g  \psi \, \Vg 
   +\int_{M}\left(\Sh^2 u_{\eps},iu_{\eps}\right)_g \psi \, \Vg = 
  \int_{M}(f_\eps, iu_\eps)_g \, \psi \, \Vg
 \]
 and thus, thanks to \cite[Lemma A.1]{Ag-Reno}, 
 \[
  \int_{M}\psi \, \d^* j(u_\eps) 
  = \int_{M}(f_\eps, iu_\eps)_g \, \psi \, \Vg 
   -\int_{M}\left(\Sh^2 u_{\eps}, iu_{\eps}\right)_g \psi \, \Vg.
 \]
 Since $\d^* j(u_\eps) = \d^*\d \theta_\eps = -\Delta \theta_\eps$, 
 we deduce that $\theta_\eps$ solves 
 \begin{equation}
 \label{eq:thetaeps}
 -\Delta \theta_\eps = \left(f_\eps-\Sh^2 u_{\eps},iu_{\eps}\right)_g
 \end{equation}
 As $f_\eps$ is bounded in~$L^2(M)$ and~$u_\eps$
 is bounded in~$L^\infty(M)$, by assumption,
 the right hand side of~\eqref{eq:thetaeps}
 is bounded in $L^2(M)$. Therefore, elliptic regularity 
 implies that $\theta_\eps$ is bounded in $H^2(M)$ and hence,
 (i) follows, by compact Sobolev embedding.
 Let~$\theta\in \Lambda$ be a cluster point for the
 sequence~$\theta_\eps$ in~$H^1(M)$. By passing to the limit
 in~\eqref{eq:thetaeps} and recalling that~$f_\eps\xrightarrow{\eps\to 0} 0$
 strongly in $L^2(M)$, we obtain
 \begin{equation}
 \label{eq:theta0}
  -\Delta \theta = -\left(\Sh^2 u_0,iu_0\right)_g.
 \end{equation}
 where~$u_0$ is a cluster point for the~$u_\eps$ in~$W^{1,p}(M)$
 (see~\eqref{convu0}). The right-hand side of Equation~\eqref{eq:theta0}
 belongs to~$L^\infty(M)$, and its~$L^\infty(M)$-norm is bounded
 in terms of~$M$, $C_0$ only. By elliptic regularity,
 we deduce that~$\theta\in W^{2,q}(M)$ for any~$q < +\infty$
 (with uniform bounds on the~$W^{2,q}(M)$-norm). 
 The property~(iii) follows from the compact 
 Sobolev embedding~$W^{2,q}(M)\hookrightarrow C^1(M)$,
 for~$2< q < +\infty$. To complete the proof of the lemma, 
 it only remains to show that~$\Lambda$ is closed in~$C^1(M)$.
 A diagonal argument shows that the set of cluster points~$\Lambda$ 
 is closed in~$H^1(M)$; as~$C^1(M)$ embeds continuously in~$H^1(M)$
 it follows that~$\Lambda$ is closed in~$C^1(M)$.
\end{proof}

\begin{remark} \label{rk:thetaeps}
 Equation~\eqref{eq:thetaeps} implies that~$(\theta_\eps)_{\eps > 0}$
 is bounded in~$H^2(M)$ and hence, relatively compact not only in~$H^1(M)$,
 but also in~$W^{1,q}(M)$ for any~$q <+\infty$.
\end{remark}

\begin{proof}[Proof of Proposition~\ref{prop:gammaliminf_extr}]
 We extract a (non-relabelled) subsequence such that~$u_\eps$
 converges to a limit vector field~$u_0$, as in~\eqref{convu0}.
 Let~$\theta_\eps$ be defined as in~\eqref{Hodge-j}--\eqref{Hodge-j-average}.
 
 \setcounter{step}{0}
 \begin{step}[Proof of~\eqref{eq:hodgeu0}]
  The convergence of $u_\eps$, Equation~\eqref{convu0},
  together with~\eqref{hp:H1bounds-Linfty-bis} implies that 
  \[
   j(u_\eps)\to j(u_0) \qquad \textrm{strongly}\quad \textrm{in}\quad L^p(M).
  \]
  Moreover, since the the projection $\mathbb{P}_H$ on the harmonic $1$-forms is linear and bounded from $L^1(M; \, T'M)$ to $\Harm^1(M)$, we get that 
  \begin{equation}
   \label{eq:proj_harmu0}
   \xi_\eps := \mathbb{P}_H j(u_\eps)\to \mathbb{P}_{H}j(u_0)=:\xi
  \end{equation}
  On the other hand, since the projection~$\mathbb{P}_{\d}$ 
  on the exact forms is linear and bounded as an operator 
  \[
   \mathbb{P}_\d\colon L^p\left(M; \, T'M\right)\to L^p(M),
   \qquad \mathbb{P}_\d (j(u_\eps)):= \d \theta_\eps,
  \]
  for $p\in (1, \, +\infty)$ (see \cite{scott95}), we obtain
  \begin{equation}
  \label{eq:proj_exactu0}
   \d \theta_\eps = \mathbb{P}_{\d} j(u_\eps)\to \mathbb{P}_{\d} j(u_0) = \d \theta,
  \end{equation}
  strongly in $L^p(M)$. 
  Finally, the convergence of~$j(u_\eps)$ and of
  the vorticity~$\omega(u_\eps)=\d j(u_\eps)  + \kappa \, \Vg$ 
  (see \eqref{hp:H1bounds-flat-bis}) allows to identify 
  \begin{equation}
  \label{eq:dju0}
   \d j(u_0) = 2\pi\sum_{j=1}^n d_j \delta_{a_j} - \kappa \, \Vg.
  \end{equation}
  Equation~\eqref{eq:hodgeu0} now follows from~\eqref{eq:proj_harmu0}, \eqref{eq:proj_exactu0} and~\eqref{eq:dju0}. Indeed,
  the Hodge decomposition for $j(u_0)$ reads, 
  thanks to \eqref{eq:proj_harmu0} and \eqref{eq:proj_exactu0},
  \[
   j(u_0) = \d \theta + \d^*\beta + \xi,
  \]
  for some $2$-form $\beta$. Due to~\eqref{eq:dju0}, the $2$-form~$\beta$
  (which can be taken with zero mean) has to satisfy 
  \[
   \d\d^*\beta = -\Delta \beta = 2\pi\sum_{j=1}^n d_j \delta_{a_j} -\kappa \, \Vg, 
  \]
  and thus $\beta=\Psi$. 
 \end{step}
 
 \begin{step}[Proof that~$\xi\in\mathcal{L}(\a, \, \db)$ and 
 that~$\theta$ is a critical point of~$\GG$]
  For any $\eps>0$ we introduce the following 
  sequence of smooth vector fields 
  \[
   u_\eps^{*}:= e^{-i\theta_\eps}u_\eps. 
  \]
  Since (see \cite[Section 10.2]{JerrardIgnat_full}) 
  \begin{align*}
   \abs{D u_\eps^*}^2_g &= \abs{D u_\eps}^2_g + \abs{u_\eps}_g^2\abs{\d \theta_\eps}^2_g -2(j(u_\eps),\d \theta_\eps)_g\\ 
   & = 
   \abs{D u_\eps}^2_g +\left(\abs{u_\eps}^2-2\right)\abs{\d \theta_\eps}^2_g,
  \end{align*}
  and $\norm{\theta_\eps}_{H^1(M)}\le C$, we conclude that 
  \begin{align}
  \label{eq:energy_boundu*}
   F_{\eps}^{\intr}(u^*_{\eps}) &= F^{\extr}_{\eps}(u_\eps) + \int_{M}\left(\abs{u_\eps}^2-2\right)\abs{\d \theta_\eps}^2_g\Vg \nonumber\\
   &\stackrel{\eqref{hp:H1bounds-Linfty-bis}-\eqref{hp:H1bounds-energy-bis}}\le n\pi \abs{\log\eps} + C.
  \end{align}
  By applying compactness results for the intrinsic Ginzburg-Landau energy
  (see, e.g., \cite[Lemma 4.3, Corollary 4.5, Lemma~4.9]{Ag-Reno}),
  we find a vector field $u^*\in W^{1,p}_{\tang}(M)$
  with $\abs{u^*}_g=1$ a.e.~in~$M$ 
  and a (non-relabelled) subsequence such that 
  \[
   u_\eps^*\to u^*\qquad \textrm{strongly in } W^{1,p}_{\tang}(M) 
   \qquad \textrm{ and in } L^q_{\tang}(M) 
  \]
  for any~$p\in (1, \, 2)$, $q\in (1, \, +\infty)$.
  This implies
  \[
   j(u_{\eps}^*)\to j(u^*) \qquad \textrm{in } L^{p}(M)
  \]
  for any~$p\in (1, \, 2)$.
  Since $u_\eps\to u_0$ in $W^{1,p}_{\tang}(M)$ and in $H^1_{\mathrm{loc}}(M\setminus\{a_1, \, \ldots, \, a_n\})$ and $\theta_\eps\to \theta$ in $H^1(M)$ we obtain  
  \[
   u_\eps^{*}\to e^{-i\theta}u_0\qquad \textrm{weakly in } W^{1,p}(M)
   \textrm{ and in } H^1_{\mathrm{loc}}(M\setminus\{a_1, \, \ldots, \, a_n\}),
  \]
  and therefore $u^* = e^{-i\theta}u_0$. By direct computation, 
  \begin{equation}
  \label{eq:currentu*}
   j(u^*) = j(u_0) -\d \theta.  
  \end{equation}
  Therefore, 
  \begin{equation}
  \label{eq:vorticityu*W11}
   \omega(u_\eps^*) \xrightarrow{\eps\to 0}\d j(u^*) +\kappa \, \Vg 
   = 2\pi\sum_{j=1}^n d_j \delta_{a_j} \qquad 
   \textrm{in } W^{-1,1}(M).
  \end{equation}
  The vector field $u^*$ satisfies $\abs{u^*}=1$ a.e.~in $M$, 
  \begin{align}
   &\d j(u^*) = 2\pi \sum_{j=1}^n d_j \delta_{a_j} -\kappa \,\Vg,\label{eq:vorticityu*}\\
   &\d^* j(u^*) = 0. 
   \label{eq:covorticityu*}
  \end{align}
  Therefore, $u^*$ is a canonical harmonic vector field for~$(\a,\db,\xi)$.
  This implies~$\xi\in \mathcal{L}(\a, \, \db)$
  (see \cite[Theorem 2.1]{JerrardIgnat_full}). Moreover,
  as~$u_0 = e^{i\theta}u^*$, the equation~\eqref{eq:theta0} is
  indeed the Euler-Lagrange equation for the energy $\GG(u^*, \, \cdot)$ 
  defined in \eqref{G}. In other words, Equation~\eqref{eq:theta0}
  may be rewritten as
  \[
   \partial_\theta \GG(u^*, \, \theta) =0.
  \]
  and hence, $\theta$ is a critical point 
  (albeit not necessarily a minimiser) of~$\GG(u^*, \, \cdot)$.
  As the canonical harmonic field for~$(\a, \, \db, \, \xi)$
  is unique up to a global rotation \cite[Theorem~2.1]{JerrardIgnat_full},
  there exists a constant~$\kappa\in\R$ such that 
  $u^* = e^{i\kappa} u^*[\a, \, \db, \, \xi]$
  (where~$u^*[\a, \, \db, \, \xi]$ is identified 
  by~\eqref{canonicalvf}, \eqref{pointvalue}.
  Upon replacing~$\theta$ with~$\theta + \kappa$,
  we may even assume without loss of generality
  that~$u^* = u^*[\a, \, \db, \, \xi]$, for rotating~$u^*$
  by a constant angle is equivalent to shifting~$\theta$
  by an additive constant (by~\eqref{G-symmetry}).
 \end{step}
 
 \begin{step}[Proof of~\eqref{eq:gammalinf_extr}]
  We have
  \begin{equation}
  \label{eq:gamma_conv1}
   \begin{split}
    F_\eps^{\extr}(u_\eps) = F_{\eps}^{\intr}(u^{*}_\eps) 
    &+ \frac{1}{2}\int_{M}\abs{u_\eps^{*}}^2_{g}\abs{\d \theta_\eps}^2_g\Vg + \frac{1}{2}\int_{M}\abs{\Sh\left(e^{i\theta_\eps} u_\eps^*\right)}^2_{g}\Vg\\
    &+ \int_{M}\left(j(u_\eps^*),\d \theta_\eps\right)_g\Vg.
   \end{split}
  \end{equation}
  First of all, we observe that~$u_\eps^*$ satisfies 
  the hypothesis of~\cite[Theorem 2.5]{JerrardIgnat_full},
  which gives
  \begin{equation}
  \label{eq:gammaliminf_intr}
   \liminf_{\eps\to 0}\left(F_{\eps}^{\intr}(u^{*}_\eps) - n\pi\abs{\log\eps}\right)
   \ge W^{\intr}(\a,\db,\xi) + n\gamma.
  \end{equation}
  Moreover, the energy estimate~\eqref{eq:energy_boundu*} implies
  \[
   \int_M\left(\abs{u_\eps^*}^2 -1\right)^2 \Vg \le C\eps\abs{\log\eps}
  \]
  and hence,
  \[
   \abs{u_\eps^*}^2\xrightarrow{\eps\to 0}1
   \qquad \textrm{ strongly in  } L^2(M)
  \]
  Keeping in mind that~$\theta_\eps\xrightarrow{\eps\to 0}\theta$
  strongly in~$H^1(M)$, we can pass to the limit in the second 
  term in \eqref{eq:gamma_conv1} and obtain 
  \begin{equation}
  \label{eq:second_term}
   \lim_{\eps\to 0}\frac{1}{2}\int_{M}\abs{u_\eps^{*}}^2_{g}\abs{\d \theta_\eps}^2_g\Vg
   = \frac{1}{2}\int_{M}\abs{\d \theta}^2_g\Vg.
  \end{equation}
  The strong $L^2_{\tang}(M)$-convergence of $u_\eps=e^{i\theta_\eps} u_\eps^*$
  to $u_0= e^{i\theta} u^*$ implies that 
  \begin{equation}
  \label{eq:third_term}
   \lim_{\eps\to 0}\int_{M}\abs{\Sh\left(e^{i\theta_\eps} u_\eps^*\right)}^2_{g}\Vg
   =\int_{M}\abs{\Sh\left(e^{i\theta} u^*\right)}^2_{g}\Vg. 
  \end{equation}
  Finally, since $\d\theta_\eps$ converges strongly in $L^q(M)$ 
  for any $q\in [1, \, +\infty)$ (see Remark~\ref{rk:thetaeps}), we have
  \begin{equation}
  \label{eq:last_term}
   \lim_{\eps\to 0}\int_M\left(j(u_\eps^*),\d \theta_\eps\right)_g\Vg = 
   \int_M\left(j(u^*),\d \theta\right)_g\Vg,
  \end{equation}
  thanks to the continuity of the product of a weakly and a strongly convergent sequence in~$L^p$. 
  Therefore, thanks to \eqref{eq:hodgeu0} and \eqref{eq:currentu*} we have that 
  \[
   \int_{M}\left(j(u^*),\d \theta\right)_g \Vg = 
   \int_{M}\left(j(u_0),\d\theta\right)_g \Vg-\int_{M}\abs{\d\theta}_g^2\Vg= 0.
  \]
  Thus, taking the $\liminf_{\eps\to 0}$ in \eqref{eq:gamma_conv1} we readily get  
  \eqref{eq:gammalinf_extr}.
  \qedhere
 \end{step}
\end{proof}

\begin{prop} \label{prop:liminfgrad}
 Let~$u_\eps\in H^1_{\tang}(M)$ be a sequence 
 of solutions of~\eqref{inhomGL}
 that satisfies~\eqref{hp:H1bounds-flat-bis}, 
 \eqref{hp:H1bounds-Linfty-bis}, \eqref{hp:H1bounds-energy-bis}.
 Assume that~$f_\eps\xrightarrow{\eps\to 0}0$ in~$L^2_{\tang}(M)$.
 Assume moreover that~$\xi_\eps:=\P_H j(u_\eps)\xrightarrow{\eps\to 0}\xi$
 and, upon extraction of a subsequence, that 
 $u_\eps\xrightarrow{\eps\to 0} u_0 := e^{i\theta} u^*[\a, \, \db, \, \xi]$
 in~$W^{1,p}(M)$, where~$u^*[\a, \, \db, \, \xi]$ 
 is given by~\eqref{canonicalvf}--\eqref{pointvalue} and~$\theta\in H^1(M)$
 is a critical point of~$\GG(u^*[\a, \, \, \db, \, \xi], \, \cdot)$.
 Then, we have
 \begin{equation*} 
  \begin{split}
   \liminf_{\eps\to 0} \frac{\abs{\log\eps}}{2}
    \int_{M}\abs{ -\Delta_g u_\eps 
     + \frac{1}{\eps^2}(\vert u_\eps\vert^2-1)u_\eps}^2_{g}\,\Vg
     \geq  \frac{1}{2\pi} \sum_{j=1}^n
     \abs{\nabla_{a_j} W(\a, \, \mathbf{d}, \, \xi, \,\theta)}_g^2
  \end{split}
 \end{equation*}
\end{prop}
\begin{proof}
 Let~$j\in\{1, \, \ldots, \, n\}$ be fixed.
 For~$\eta>0$ small enough,
 we consider the geodesic ball $B_\eta(a_j)(a_j)$ centered at $a_j$. 
 We denote with $\nu$ its outward normal. 
 By reasoning as in~\cite[Proposition 5.9]{Ag-Reno} it is not difficult to obtain that 
\begin{equation} \label{eq:liminfgrad}
 \begin{split}
  L_j &:= \liminf_{\eps\to 0}\frac{\vert \log\eps\vert}{2}\int_{B_\eta(a_j)}\abs{ -\Delta_g u_\eps + \frac{1}{\eps^2}(\vert u_\eps\vert_g^2-1)u_\eps + \Sh^2 u_\eps}^2_{g}\,\Vg \\
  &\ge  \frac{1}{2\pi}\abs{\frac{1}{2}\int_{\partial B_\eta(a_j)}\vert D u_0\vert^2_g (\nu,e_k)_g\, \d\H^1 - \int_{\partial B_\eta(a_j)}(D_{e_k} u_{0}, D_{\nu}u_0)_g\right.\, \d\H^1
  \\
  &\qquad\qquad + \left.\int_{B_\eta(a_j)}(D u_0,R(\cdot,e_k)u_0)_g\,\Vg + \int_{B_\eta(a_j)}(\Sh^2 u_0,D u_0)_g \,\Vg}^2,
 \end{split}
\end{equation}
where $R$ denotes the Riemannian curvature tensor (see \cite[Lemma A.3]{Ag-Reno}).
Therefore, since the vector field $u_0\in W^{1,p}_{\tang}(M)\cap L^\infty_{\tang}(M)$ for $p\in [1,2)$ and the curvature tensor and the shape operator are smooth, we readily have that   
\[
\lim_{\eta\to 0}\left(\int_{B_\eta(a_j)}(D u_0,R(\cdot,e_k)u_0)_g  \,\Vg +\int_{B_\eta(a_j)}(\Sh^2 u_0,D u_0)_g \,\Vg\right) =0. 
\] 
Thanks to Proposition~\ref{prop:gradient_reno} 
we have
\begin{equation*}
 \begin{split}
  L_j \ge \frac{1}{2\pi} \abs{\left(\nabla_{a_j} W(\a, \, \db, \, \xi,\,\theta),
  \, {\hat{\be}}_k\right)_g}^2 + \mathrm{o}_{\eta\to 0}(1)
 \end{split}
\end{equation*}
As~$\{\hat{\be}_1, \, \hat{\be}_2\}$ is an arbitrary orthonormal basis
for~$\T_{a_j}(M)$, we also deduce
\begin{equation*}
 \begin{split}
  L_j \ge \frac{1}{2\pi} \abs{\left(\nabla_{a_j} W(\a, \, \db, \, \xi,\,\theta),
  \, {\hat{\be}}\right)_g}^2 + \mathrm{o}_{\eta\to 0}(1)
 \end{split}
\end{equation*}
for any unit vector~$\hat{\be}\in\T_{a_j}(M)$. By taking the
supremum over~$\hat{\be}$, we obtain
\begin{equation*}
 \begin{split}
  L_j \ge \frac{1}{2\pi} \abs{\nabla_{a_j} W(\a, \, \db, \, \xi,\,\theta)}_g^2 + \mathrm{o}_{\eta\to 0}(1)
 \end{split}
\end{equation*}
 By taking the sum over~$j$, and letting~$\eta\to 0$,
 the proposition follows.
\end{proof}

\subsection{Vortex Dynamics: Proof of Theorem \ref{th:main1}}

We consider {\itshape well-prepared initial conditions}.
To ease the reading, we recall the definition.
Given~$n\in\Z$, $n\geq 1$, we 
consider~$(\a^0, \, \db)\in\mathscr{A}^n$ 
such that $d_j=\pm 1$ for any $j=1,\ldots, n$
and $\xi^0\in \mathcal{L}(\a^0, \, \db)$.
Finally, we let $\theta^0\in H^1(M)$ 
We assume that the initial conditions
$u^0_{\eps}\in H^1_{\tang}(M)$ satisfy
\begin{align}
 & \omega (u^{0}_\eps)\xrightarrow{\eps \to 0} 2\pi \sum_{j=1}^n d_j \delta_{a_j^{0}} \qquad \hbox{ in } W^{-1,p}(M) \quad
 \textrm{for any } p\in (1, \, 2) \label{eq:initial_vorticity}\\
 & F_\eps^{\extr}(u_\eps^{0}) = \pi n \abs{\log\eps}
 + W(\a^0, \, \db, \, \xi^0, \, \theta^0)
 + n\gamma + \mathrm{o}_{\eps\to 0}(1)\label{eq:well_prepared}, \\
 &  \norm{u^0_\eps}_{L^\infty(M)} \leq 1.\label{eq:initial_Linfty}
\end{align} 
As we did in Section~\ref{sect:G},
we fix a sufficiently small neighbourhood~$\mathscr{U}$
of~$\a^0$ in~$M^n$, a point~$b_0\in M$ such that $b^0\neq a_k$ for 
any~$\a = (a_1, \, \ldots, \, a_n)\in\mathscr{U}$ and any index~$k$,
and a tangent vector~$w^0\in T_{b^0} M$. 
For any~$\a\in\mathscr{U}$ and any~$\xi\in\mathcal{L}(\a, \, \db)$
we consider the reference canonical harmonic 
vector field~$u^*[\a, \, \db, \, \xi]$ 
that satisfies $u^*[\a, \, \db, \, \xi](b^0) = w^0$.

The existence of a smooth solution~$u_\eps$ of~\eqref{eq:GL_intro}, 
for any $\eps>0$, is a consequence of standard parabolic theory. 
Moreover, $u_\eps$ satisfies the following a priori estimates
(see e.g.~\cite[Lemma 5.6]{Ag-Reno} for the proof). 
\begin{lemma}
\label{lem:lemma1}
Let $u_\eps$ be a solution of \eqref{eq:GL_intro} with $u_{\eps}^0$ satisfying 
\eqref{eq:well_prepared} and~\eqref{eq:initial_Linfty}.
Then, 
\begin{gather}
 \norm{u_\eps}_{L^\infty(M\times(0,T))}\le 1, \label{eq:boundLinfty} \\[5pt]
 F_\eps^{\extr}(u_\eps(t)) \le \pi n \abs{\log\eps} + C
  \qquad \hbox{ for a.a. }t\in [0,T] \label{eq:energy_bound1_lemma} \\
 \int_{0}^{T} \int_{M}\vert \partial_t u_\eps\vert^2_g \, \Vg \, \d t \le C\vert\log\eps\vert \label{eq:boundut_lemma}
\end{gather}
for some constant $C>0$ independent of $\eps$. 
\end{lemma}

\begin{prop}
\label{prop:main_limit_proc1}
 Let $u_\eps $ be a solution of \eqref{eq:GL_intro} with $u_{\eps}^0$ satisfying~\eqref{eq:initial_vorticity}, \eqref{eq:well_prepared} and~\eqref{eq:initial_Linfty}.
 Then, there exist $T^*\in (0,T]$, a curve 
 ${\bf a}=(a_1,\ldots,a_n)\in H^1(0,T^*;M^n)$
 with~$\a(t)\in\mathscr{U}$ for any~$t\in [0, \, T^*)$,
 integers~$\mathbf{d} = (d_1, \, \ldots, \, d_n)\in\Z^n$
 with~$(\a(t),\db)\in\mathscr{A}^n$ for any $t\in [0,T^*)$,
 a curve $\xi\in H^1(0,T^*;\Harm^1(M))$,
 a non-increasing function $\varphi:[0,T)\to \R$
 and a non-relabelled subsequence $\eps\to 0$
 such that, for any $t\in [0,T^*)$,
 \begin{align}
  \label{eq:conv_vorticity_prop}
  & \omega(u_\eps(t)) \to 2\pi\sum_{j=1}^n d_j\delta_{a_j(t)}\qquad \hbox{ in } W^{-1,p}(M) \quad
  \textrm{for any } p\in (1, \, 2),\\
  & \P_H j(u_\eps(t)) \xrightarrow{\eps\to 0} \xi(t),\qquad\hbox{ with } \qquad\xi(t)\in \mathcal{L}(\a(t),\db),
  \label{eq:convergence_xi}\\
  &\liminf_{\eps\to 0}\frac{1}{\vert \log\eps\vert}\int_{0}^t \int_{M}\vert \partial_t u_\eps\vert^2_g \,  \Vg\d t\ge \pi
  \int_{0}^t \vert {\bf a}'(s)\vert^2_g \, \d s,
  \label{eq:liminfut} \\
  &\lim_{\eps\to 0}\left(F_\eps^{\extr}(u_\eps(t))-\pi\abs{\log\eps} -n\gamma\right) = \varphi(t).\label{eq:conv_energy}
 \end{align}
 Moreover, there exists a measurable map
 $\theta\colon (0, \, T^*)\to C^1(M)$ and, for a.e.~$t\in (0, \, T^*)$,
 a subsequence~$\eps_h(t)\to 0$ (possibly depending on~$t$)
 such that
 \begin{equation}
  u_{\eps_{h}(t)}(t)\to u_0(t) := e^{i\theta(t)} u^*[\a(t), \, \db, \, \xi(t)]
 \end{equation}
 strongly in~$W^{1,p}_{\tang}(M)$ for any~$p\in [1, \, 2)$.
 The function~$\theta(t)$ satisfies
 \begin{equation} \label{eq:euler_theta}
  -\Delta \theta(t) = -\left(\Sh(u_0(t)),\Sh(i u_0(t))\right) 
  \qquad \hbox{a.e. in } M 
 \end{equation}
 for a.e.~$t\in (0, \, T^*)$.
 Finally, for almost any $t\in [0,T^*)$ we have
 \begin{gather}
 \varphi(t) \ge
  W(\a(t), \, \db, \, \xi(t), \, \theta(t)) \label{eq:liminf_E} \\
 \begin{split}
   &\liminf_{\eps\to 0} \frac{\abs{\log\eps}}{2}
    \int_{M}\abs{ -\Delta_g u_\eps(t) 
     + \frac{1}{\eps^2}(\vert u_\eps(t)\vert^2-1)u_\eps(t)}^2_{g}\,\Vg \\
   &\hspace{3.3cm} \geq \frac{1}{2\pi} \sum_{j=1}^n
     \abs{\nabla_{a_j} W(\a(t), \, \mathbf{d}, \, \xi(t), \,\theta(t))}_g^2
  \end{split}  \label{eq:liminf_grad_E}
 \end{gather}
\end{prop}
The inferior limit in~\eqref{eq:liminf_grad_E}
is taken not only with respect to the subsequence~$\eps_h(t)\to 0$,
but also with respect to the sequence~$\eps\to 0$
(that does not depend on~$t$).
\begin{proof}
\setcounter{step}{0}

We split the proof into steps.

\begin{step}
Since $F_\eps(u)\le F_\eps^{\extr}(u)$ for any $u\in H^1_{\tang}(M)$,
the sequence $u_\eps$ satisfies the hypothesis of~\cite[Proposition 5.7]{Ag-Reno}.
Therefore, there exist a time $T^*\in (0,T]$, 
a curve 
${\bf a}=(a_1,\ldots,a_n)\in H^1(0,T^*;M^n)$, 
integers~$\mathbf{d} = (d_1, \, \ldots, \, d_n)\in\Z^n$
and a curve $\xi\in H^1(0,T^*;\Harm^1(M))$ satisfying \eqref{eq:conv_vorticity_prop}-\eqref{eq:liminfut}. 
The map~$t\mapsto \a(t)$ is continuous, by Sobolev embedding.
Therefore, by taking~$T^*$ small enough, we may assume without 
loss of generality that~$\a(t)\in\mathscr{U}$ for any~$t\in [0, \, T^*)$.
Moreover, since~\eqref{eq:GL_intro} is a gradient flow, the function
\[
t\mapsto  F_\eps^{\extr}(u_\eps(t))-\pi\abs{\log\eps} -n\gamma\qquad t\in [0,T^*) 
\]
is non increasing, for any~$\eps> 0$. 
Therefore, thanks to Helly's selection Theorem (see \cite[Lemma 3.3.3]{Amb-Gi-Sav}) there exists a subsequence of $\eps$ and a non increasing function $\varphi:[0,T^*)\to \R$ such that 
\begin{equation}
\label{eq:helly}
 \lim_{\eps\to 0} \left(F_\eps^{\extr}(u_\eps(t))
  -\pi\abs{\log\eps} - n\gamma\right) 
 = \varphi(t) \qquad \textrm{for any } t\in [0,T^*).
\end{equation}
Moreover, since the initial conditions are well-prepared we have that 
\[
 \varphi(0) 
 = W(\a^0, \, \db, \, \xi^0, \, \theta^0) 
\]
\end{step}

\begin{step}
We have
\[
\int_{0}^T\int_{M}\abs{\Delta_g u_\eps + \frac{1}{\eps^2}(\vert 1- \vert u_\eps\vert^2)u_\eps + \Sh^{2} u_\eps}^2_{g} \Vg \d t = \frac{1}{\vert\log\eps\vert^2}\int_{0}^{T}\int_M\vert \partial_t u_\eps\vert^2_g \, \Vg\d t,
\]
and hence, thanks to \eqref{eq:boundut_lemma},
\[
\int_{0}^T \int_{M}\abs{-\Delta_g u_\eps + \frac{1}{\eps^2}\left( \abs{u_\eps}_g^2-1\right)u_\eps +\Sh^{2} u_\eps }_{g}^2 \, \Vg \d t \le \frac{C}{\vert \log\eps\vert}.
\]
Therefore 
\begin{equation}
\label{eq:Deltazero}
-\Delta_g u_\eps + \frac{1}{\eps^2}\left( \abs{u_\eps}_g^2-1\right)u_\eps +\Sh^{2} u_\eps \to 0\,\,\,\,\hbox{ in }L^2(0,T;L^2(M))
\end{equation}
and, up to subsequence, 
\begin{equation}
\label{eq:Deltazero_point}
-\Delta_g u_\eps + \frac{1}{\eps^2}\left( \abs{u_\eps}_g^2-1\right)u_\eps +\Sh^{2} u_\eps \to 0 \quad \hbox{ in } L^2_{\tang}(M)\quad \hbox{for a.e. } t\in (0,T^*). 
\end{equation}
We let $C\subseteq (0,T^*)$ be the set of those $t$ that satisfy~\eqref{eq:Deltazero_point}.
\end{step}

\begin{step} \label{step:Lambda}
 We consider the $L^2$-orthogonal projection
 of~$j(u_\eps(t))$ onto the subspace of exact $1$-forms,
 i.e.~$\P_\d j(u_\eps(t) = \d\theta_\eps(t)$.
 We select a unique primitive~$\theta_\eps(t)$ by imposing
 that~$\theta_\eps(t)$ integrates to zero over~$M$.
 We define a multi-valued map $\Lambda\colon [0, \, T^*)\to 2^{C^1(M)}$
 as follows: for any~$t\in [0, \, T^*)$ and any~$\theta\in C^1(M)$,
 we will say that $\theta\in\Lambda(t)$ if and only if 
 there exists a subsequence $\eps_h(t)\to 0$ such that
 $\theta_{\eps_h(t)}\to \theta$ strongly in~$H^1(M)$.
 Lemma~\ref{lemma:thetaeps} implies that~$\Lambda(t)$
 is a non-empty, compact subset of~$C^1(M)$ for any~$t\in C$.
 
 We claim that the multi-valued map~$\Lambda$ 
 is \emph{measurable} --- that is, for any 
 closed set~$K\subseteq C^1(M)$, the `inverse image'
 \[
  \Lambda^{-1}(K) := \{t\in [0, \, T^*)\colon \Lambda(t)\cap K\neq\emptyset\}
 \]
 is measurable. By Lemma~\ref{lemma:thetaeps},
 there exists a compact subset~$K_*\subseteq C^1(M)$
 such that~$\Lambda(t)\subseteq K_*$ for any~$t\in C$.
 Therefore, if we show that~$\Lambda^{-1}(K)$ is measurable 
 for any \emph{compact} subset~$K\subseteq C^1(M)$, then it will
 follow that $\Lambda^{-1}(K)$ is measurable
 for any \emph{closed} subset~$K\subseteq C^1(M)$.
 Let~$K$ be a compact subset of~$C^1(M)$. The set~$K$ is
 compact in~$H^1(M)$ too; in particular, $K$ is closed in~$H^1(M)$.
 For any~$\theta\in C^1(M)$, let
 \[
  \dist_K(\theta) := \inf_{u\in K} \norm{\theta - u}_{H^1(K)}
 \]
 If~$\Lambda(t)\cap K \neq \emptyset$ then
 $\liminf_{\eps\to 0}\dist_K(\theta_\eps(t)) = 0$, by definition of~$\Lambda(t)$. 
 Conversely, assume that~$t\in C$ 
 and~$\liminf_{\eps\to 0}\dist_K(\theta_\eps(t)) = 0$. 
 As the sequence~$\theta_\eps(t)$ is relatively compact
 in~$H^1(M)$ when~$t\in C$ (by Lemma~\ref{lemma:thetaeps}), 
 it follows that there exists a subsequence~$\theta_{\eps_h(t)}(t)$ 
 that converges $H^1(M)$-strongly to an element of~$K$. 
 As a consequence, the set~$\Lambda^{-1}(K)$ coincides with
 $\{t\in [0, \, T^*)\colon \liminf_{\eps\to 0}
 \dist_K(\theta_\eps(t)) = 0\}$,
 up to Lebesgue-negligible sets. As~$\theta_\eps\colon [0, \, T^*)\to C^1(M)$
 is continuous for any~$\eps$, it follows that $\Lambda^{-1}(K)$ is measurable.
\end{step}

\begin{step}
 Let $U\colon [0, \, T^*)\times C^1(M)\to\R$
 be the function defined by
 \[
  \begin{split}
   U(t, \, \psi) := \abs{\nabla_{\a} W(\a(t), \, \db, \, \xi(t), \, \psi)}
   = \abs{\nabla_{\a} W^{\intr}(\a(t), \, \db, \, \xi(t)) 
    + 2\pi \sum_{j=1}^n d_j \, i \nabla \psi(a_j(t))} 
  \end{split}
 \]
 for any~$(t, \, \psi)\in [0, \, T^*)\times C^1(M)$.
 The function~$U$ is continuous, because~$\a\in H^1(0, \, T^*; \, M^n)$
 is continuous (by Sobolev embedding). We know (by Lemma~\ref{lemma:thetaeps})
 that the set~$\Lambda(t)$ is compact and non-empty, for any~$t\in C$.
 Therefore, for any~$t\in C$, there exists~$\theta(t)\in\Lambda(t)$
 such that 
 \begin{equation} \label{minimaltheta}
  U(t, \, \theta(t)) = \min\left\{U(t, \, \psi)\colon \psi\in\Lambda(t)\right\}
 \end{equation}
 As~$\Lambda$ is measurable, we can choose~$\theta(t)$
 in such a way that the map~$\theta\colon[0, \, T^*)\to C^1(M)$
 is measurable (see e.g.~\cite[Theorem~9.1.(iii)]{Wagner-Survey}).
 
 We fix~$t\in C$. As~$\theta(t)\in\Lambda(t)$, there exists 
 a subsequence $\eps_h(t)\xrightarrow{h\to +\infty} 0$,
 which may depend on~$t$, such that $\theta_{\eps_h(t)}(t)\to\theta(t)$
 strongly in~$H^1(M)$ as~$h\to+\infty$.
 By~\cite[Lemma~4.9]{Ag-Reno}, we may extract a further
 subsequence (still denoted~$\eps_h(t)$) and find 
 a vector field $u_0(t)\in W^{1,p}_{\tang}(M)$
 such that 
 \begin{equation*}
  u_{\eps_h(t)}(t) \to u_0(t) \qquad \textrm{strongly in }
  W^{1,p}_{\tang}(M)
 \end{equation*}
 for any~$p\in [1, \, 2)$. 
 By Proposition~\ref{prop:gammaliminf_extr},
 $u_0(t)$ must have the form~$u_0(t) = e^{i\theta(t)} u^*(t)$
 where~$u^*(t)$ is a canonical harmonic field for~$(\a(t), \, \db, \, \xi(t))$.
 Up to modifying~$\theta(t)$ by an additive constant,
 we may further assume that~$u^*(t)=u^*[\a(t), \, \db, \, \xi(t)]$
 where~$u^*[\a(t), \, \db, \, \xi(t)]$ 
 satisfies~\eqref{canonicalvf}--\eqref{pointvalue}\footnote{By 
 applying Kuratowski and Ryll-Nardzewski's 
 selection theorem (see~e.g.~\cite[Theorem~4.1]{Wagner-Survey}), 
 along the lines of Step~\ref{step:Lambda},
 we can make sure that the function $t\mapsto u_0(t, \, b^0)$
 is measurable (where~$b^0\in M$ is the same reference point
 as in~\eqref{pointvalue}). It follows that 
 $u^*(t) = e^{i\kappa(t)} u^*[\a(t), \, \db, \, \xi(t)]$
 where the function~$\kappa\colon [0, \, T^*)\to\R$
 is measurable, because it is completely
 determined by the values of~$u^*(t)$ and~$u^*[\a(t), \, \db, \, \xi(t)]$
 at the point~$b^0$. Therefore, the map~$t\mapsto\tilde{\theta}(t) 
 := \theta(t) + \kappa(t)$ is still measurable.}.
 The proposition follows, by
 taking into account the results of Section~\ref{sect:inhomGL}
 --- in particular, \eqref{eq:liminf_E} follows from Proposition~\ref{prop:gammaliminf_extr},
 while~\eqref{eq:liminf_grad_E} follows from
 Proposition~\ref{prop:liminfgrad} and~\eqref{minimaltheta}.
 \qedhere
\end{step}

\end{proof}

\subsubsection{Proof of Theorem \ref{th:main1}}

We can finally prove our main result, Theorem \ref{th:main1},
which is a consequence of Propositions \ref{prop:gammaliminf_extr},
\ref{prop:liminfgrad} and~\ref{prop:main_limit_proc1}.

First of all, we recall that equation \eqref{eq:GL_intro} is the $L^2$-gradient flow of the extrinsic Ginzburg-Landau energy $F_\eps^{\extr}$. 
More precisely, for any $\eps>0$ and for any $v\in H^1_{\tang}(M)$ we let 
\[
e_\eps(v):=\frac{1}{2}\left(\abs{D v}_g^2 + \abs{\Sh v}^2_g + \frac{1}{2\eps^2}(1-\vert v\vert^2)^2\right),
\]
and we define
\[
\displaystyle E_\eps(v):=
\begin{cases}
F_\eps^{\extr}(v) -\pi n\vert \log\eps\vert-n\gamma 
 &\hbox{if } e_\eps(v)\in L^1(M)\\[3pt]
+\infty &\hbox{otherwise in }L^2_{\tang}(M). 
\end{cases}
\]
Then we set $X_\eps:= L^2_{\tang}(M)$ endowed with the norm $\norm{v}_{X_\eps}:=\frac{1}{\vert \log\eps\vert^{1/2}}\norm{v}_{L^2(M)}$.
It is easy to check that the subdifferential of $E_\eps$ (with respect to the scalar product in $X_\eps$) is singlevalued and is given by  
\[
\partial_{X_\eps} E_\eps(v) =\abs{\log\eps}\left(-\Delta v + \frac{1}{\eps^2}(\vert v\vert_g^2 -1)v +\Sh^2 v\right).
\]
Therefore the evolution \eqref{eq:GL_intro} rewrites as the gradient flow
\[
\begin{cases}
\partial_t u_\eps(t) + \partial_{X_\eps}E_\eps(u_\eps) = 0 \\[3pt]
u_\eps(0) = u_{\eps}^0.
\end{cases}
\]
Equivalently, $u_\eps$ is a {\itshape curve of maximal slope} for the energy $E_\eps$ with respect to the slope $\| \partial_{X_\eps}E_\eps\|_{X_\eps}$ and thus it satisfies, for any $0\le s\le t <T^*$
\begin{align}
\label{eq:gradfloweps}
\frac{1}{2}\int_{s}^t\norm{\partial_t u_\eps(\tau)}_{X_\eps}^2 \d \tau
 + \frac{1}{2}\int_{s}^{t}\norm{\partial_{X_\eps}E_\eps(u_\eps(\tau))}^2_{X_\eps} \d \tau + E_\eps(u_\eps(t)) = E_\eps(u_\eps(s)).
\end{align}
The solution $u_\eps$ verifies the hypothesis of Proposition \ref{prop:gammaliminf_extr} at any fixed $t$. 
Therefore, there exists a non increasing function $\varphi$ such that for any $t\in (0,T^*)$ (see \eqref{eq:conv_energy}),
\begin{equation}
\label{eq:liminf_E1}
\lim_{\eps\to 0}E_\eps(u_\eps(t)) = \varphi(t),
\end{equation}
and (see~\eqref{eq:well_prepared})
\[
\varphi(0) 
= W(\a^0, \, \db, \, \xi^0, \, \theta^0)
\]
Moreover for almost any $t\in (0,T^*)$, Equation~\eqref{eq:liminf_E} reads
\[
\varphi(t) \ge W(\a(t), \, \db, \, \xi(t), \, \theta(t)). 
\]
where the function $\theta(t)\in C^{1}(M)$ is such that 
\[
\partial_{\theta}\GG(u^*[\a(t), \, \db, \, \xi(t)], \, \theta(t)) = 0
\]
(see~\eqref{eq:euler_theta}).
Then, Proposition \ref{prop:liminfgrad} and Fatou's Lemma  imply 
that for any $s$, $t\in [0,T^*)$ with $s\le t$ there holds
\begin{align*}
\liminf_{\eps\to 0}\int_{s}^{t}\norm{\partial_{E_\eps} u_\eps(s)}_{X_\eps}^2
\d s &\ge \frac{1}{2\pi}\int_{s}^t \abs{\nabla_{\a}W(\a(\tau), \, \db, \, \xi(\tau),\,\theta(\tau))}_g^2 \d \tau
\end{align*}
Moreover, recalling Proposition \ref{prop:main_limit_proc1} (see in particular \eqref{eq:liminfut}), we have 
\[
\liminf_{\eps\to 0}\frac{1}{2}\int_{s}^t \norm{\partial_t u_\eps(\tau)}^2_{X_\eps}\d \tau \ge \frac{\pi}{2} \int_{s}^t \abs{\a'(\tau)}^2_{g}\d \tau. 
\]
Therefore, from \eqref{eq:gradfloweps} we deduce that for almost any $t,s\in [0,T^*)$ with $s\le t$ there holds 
\begin{equation*}
\begin{split}
&\frac{\pi}{2} \int_{s}^t \abs{\a'(\tau)}^2_{g}\d \tau + \frac{1}{2\pi}\int_{s}^t \abs{\nabla_{\a}W(\a(\tau), \, \db, \, \xi(\tau),\,\theta(\tau))}_g^2 \d \tau + \varphi(t) \le \varphi(s),\\
& \varphi(t)\ge W(\a(t), \, \db, \, \xi(t),\,\theta(t)),\qquad \varphi(0) = W(\a_0, \, \db, \, \xi_0,\,\theta_0)
\end{split}
\end{equation*}
that is \eqref{eq:gradflowW_mainth}.
Note that the function $\varphi$ is a BV function by construction. Moreover, thanks to the above inequality it satisfies, for almost any $t\in (0,T^*)$,
\[
-\frac{\d}{\d t}\varphi(t) \ge \frac{\pi}{2}\abs{\a'(t)}^2_{g} + \frac{1}{2\pi}\abs{\nabla_{\a}W(\a(t), \, \db, \, \xi(t),\,\theta(t))}_g^2
\] 
This completes the proof of Theorem~\ref{th:main1}.

\section*{Acknowledgements}
{ \small GC \& AS are members of the GNAMPA (Gruppo Nazionale per l'Analisi Matematica, la Probabilit\`a e le loro Applicazioni)
group of INdAM.  
AS acknowledges the partial support of the MIUR-PRIN Grant 2017 ''Variational methods for stationary and evolution problems with singularities and interfaces''.
}

\bibliographystyle{plain}
\bibliography{vorticesGL}

\end{document}